\documentclass[a4paper, reqno, 10pt]{amsart}

\usepackage{a4wide}
\usepackage{verbatim}
\usepackage[dvips]{color}
\usepackage[usenames,dvipsnames]{xcolor}
\usepackage{amsmath}
\usepackage{amssymb}
\usepackage{amsfonts}
\usepackage{dsfont}
\usepackage[active]{srcltx}
\usepackage[colorlinks,pdfpagelabels,pdfstartview = FitH,bookmarksopen = true,bookmarksnumbered = true,linkcolor = NavyBlue,plainpages = false,hypertexnames = false,citecolor = OliveGreen] {hyperref}
\usepackage[italian, textsize=tiny, color=BlueGreen, backgroundcolor=BlueGreen, linecolor=BlueGreen]{todonotes}
\usepackage{mathtools}
\usepackage{aligned-overset}
\usepackage{empheq}
\usepackage{kotex}
\newcommand\blfootnote[1]{%
  \begingroup
  \renewcommand\thefootnote{}\footnote{#1}%
  \addtocounter{footnote}{-1}%
  \endgroup
}

\allowdisplaybreaks

\title[Mixed local and nonlocal double phase functionals]{Regularity results for mixed local and nonlocal double phase functionals}

\author[Byun]{Sun-Sig Byun}
\address{Department of Mathematical Sciences and Research Institute of Mathematics,
Seoul National University, Seoul 08826, Korea}
\email{byun@snu.ac.kr}

\author[Lee]{Ho-Sik Lee}
\address{Department of Mathematical Sciences,
Seoul National University, Seoul 08826, Korea}
\email{lshnsu92@snu.ac.kr}

\author[Song]{Kyeong Song}
\address{Department of Mathematical Sciences,
	Seoul National University, Seoul 08826, Korea}
\email{kyeongsong@snu.ac.kr}

\makeatletter
\@namedef{subjclassname@2020}{\textup{2020} Mathematics Subject Classification}
\makeatother
\subjclass[2020]{Primary: 49N60; Secondary: 35R11, 47G20, 35B65, 35R05}
 \keywords{Mixed local and nonlocal functionals, Double phase, Local boundedness, H\"{o}lder continuity, Harnack's inequality}

\newtheorem{theorem}{Theorem}[section]

\newtheorem{lemma}[theorem]{Lemma}

\theoremstyle{definition}
\newtheorem{definition}[theorem]{Definition}
\newtheorem{remark}[theorem]{Remark}

\numberwithin{equation}{section}

\def\eqn#1$$#2$${\begin{equation}\label#1#2\end{equation}}
\def\charfn_#1{{\raise1.2pt\hbox{$\chi_{\kern-1pt\lower3pt\hbox{{$\scriptstyle#1$}}}$}}}

\makeatletter
\newcommand{\pushright}[1]{\ifmeasuring@#1\else\omit\hfill$\displaystyle#1$\fi\ignorespaces}
\newcommand{\pushleft}[1]{\ifmeasuring@#1\else\omit$\displaystyle#1$\hfill\fi\ignorespaces}
\makeatother

\def\ep{\varepsilon}

\newcommand{\tail}{{\mathrm{Tail}}}

 \DeclareMathOperator*{\osc}{osc}

\newcommand{\data}{\mathtt{data}}

\def\er{\mathbb R}
\newcommand{\ern}{\mathbb{R}^n}

\def\loc{{\operatorname{loc}}}

\newcommand{\supp}{{\rm supp}}

\def\mean#1{\mathchoice%
          {\mathop{\kern 0.2em\vrule width 0.6em height 0.69678ex depth -0.58065ex
                  \kern -0.8em \intop}\nolimits_{\kern -0.4em#1}}%
          {\mathop{\kern 0.1em\vrule width 0.5em height 0.69678ex depth -0.60387ex
                  \kern -0.6em \intop}\nolimits_{#1}}%
          {\mathop{\kern 0.1em\vrule width 0.5em height 0.69678ex
              depth -0.60387ex
                  \kern -0.6em \intop}\nolimits_{#1}}%
          {\mathop{\kern 0.1em\vrule width 0.5em height 0.69678ex depth -0.60387ex
                  \kern -0.6em \intop}\nolimits_{#1}}}

\def\vintslides_#1{\mathchoice%
          {\mathop{\kern 0.1em\vrule width 0.5em height 0.697ex depth -0.581ex
                  \kern -0.6em \intop}\nolimits_{\kern -0.4em#1}}%
          {\mathop{\kern 0.1em\vrule width 0.3em height 0.697ex depth -0.604ex
                  \kern -0.4em \intop}\nolimits_{#1}}%
          {\mathop{\kern 0.1em\vrule width 0.3em height 0.697ex depth -0.604ex
                  \kern -0.4em \intop}\nolimits_{#1}}%
          {\mathop{\kern 0.1em\vrule width 0.3em height 0.697ex depth -0.604ex
                  \kern -0.4em \intop}\nolimits_{#1}}}

\newcommand{\aveint}[2]{\mathchoice%
          {\mathop{\kern 0.2em\vrule width 0.6em height 0.69678ex depth -0.58065ex
                  \kern -0.8em \intop}\nolimits_{\kern -0.45em#1}^{#2}}%
          {\mathop{\kern 0.1em\vrule width 0.5em height 0.69678ex depth -0.60387ex
                  \kern -0.6em \intop}\nolimits_{#1}^{#2}}%
          {\mathop{\kern 0.1em\vrule width 0.5em height 0.69678ex depth -0.60387ex
                  \kern -0.6em \intop}\nolimits_{#1}^{#2}}%
          {\mathop{\kern 0.1em\vrule width 0.5em height 0.69678ex depth -0.60387ex
                  \kern -0.6em \intop}\nolimits_{#1}^{#2}}}

\newtoks\by
\newtoks\paper
\newtoks\book
\newtoks\jour
\newtoks\yr
\newtoks\pages
\newtoks\vol
\newtoks\publ

\def\ota{{\hbox{\bf ???}}}
\def\cLear{\by=\ota\paper=\ota\book=\ota\jour=\ota\yr=\ota
\pages=\ota\vol=\ota\publ=\ota}
\def\endpaper{\the\by, \textit{\the\paper},
{\the\jour} \textbf{\the\vol} (\the\yr), \the\pages.\cLear}
\def\endbook{\the\by, \textit{\the\book},
\the\publ, \the\yr.\cLear}
\def\endpap{\the\by, \textit{\the\paper}, \the\jour.\cLear}
\def\endproc{\the\by, \textit{\the\paper}, \the\book, \the\publ,
\the\yr, \the\pages.\cLear}

\begin{document}
\maketitle
\begin{abstract}
We investigate the De Giorgi-Nash-Moser theory for minimizers of mixed local and nonlocal functionals modeled after
\[ v \mapsto \int_{\mathbb{R}^{n}}\int_{\mathbb{R}^{n}}\dfrac{|v(x)-v(y)|^{p}}{|x-y|^{n+sp}}\,dxdy+\int_{\Omega}a(x)|Dv|^{q}\,dx, \]
where $0<s<1<p \le q$ and $a(\cdot) \ge 0$. 
In particular, we prove H\"older regularity and Harnack's inequality under possibly sharp assumptions on $s,p,q$ and $a(\cdot)$.
\end{abstract}

\blfootnote{S. Byun was supported by NRF-2022R1A2C1009312. H. Lee
was supported by NRF-2020R1C1C1A01013363. K. Song was supported by NRF-2021R1A4A1027378}

\section{Introduction}
We study a mixed local and nonlocal functional whose prototype is 
\begin{align}\label{eq:ftnal0}
\mathcal{E}_0(v;\Omega) \coloneqq \iint_{\mathcal{C}_{\Omega}}\dfrac{|v(x)-v(y)|^{p}}{|x-y|^{n+sp}}\,dxdy+\int_{\Omega}a(x)|Dv|^{q}\,dx,
\end{align}
where $\Omega\subset\ern$ $(n\geq 2)$ is a bounded domain,
\begin{align*}
\mathcal{C}_{\Omega} \coloneqq (\ern\times\ern)\setminus((\ern\setminus\Omega)\times(\ern\setminus\Omega)),
\end{align*}
and 
\begin{align}\label{eq:spq}
s \in (0,1), \quad 1<p \le q.
\end{align}
Here, the modulating coefficient function $a:\Omega\rightarrow\er$ is measurable and bounded such that
\begin{align}\label{eq:a}
0\leq a(x)\leq\|a\|_{L^{\infty}}\quad (x\in\Omega).
\end{align}

The problem under consideration exhibits a purely nonlocal feature on $\{ a(x)=0 \}$, but both local and nonlocal features on $\{ a(x) > 0 \}$.
Note that in light of \eqref{eq:spq}, the $W^{1,q}$-energy has a higher regularization effect than that of the $W^{s,p}$-energy.
In this respect, our problem is deeply related to the following two topics which have been intensively studied currently: mixed local and nonlocal problems and double phase problems.

Mixed local and nonlocal linear operators, like $-\triangle + (-\triangle)^{s}$, naturally appear in the L\'evy process; in particular, nonlocal operators, like $(-\triangle)^{s}$, are related to a special case of the so-called purely jump process. 
There have been many researches on regularity for elliptic and parabolic equations involving these two kinds of linear operators; we refer to \cite{CCV11,CS07,CS09} for nonlocal equations and to \cite{BBCK09,CKSV10, CKSV12, CK10, Fo09} for mixed local and nonlocal equations, respectively.
Note that in those papers, Harnack's inequality is proved for a solution which is nonnegative in the whole domain $\mathbb{R}^{n}$. This positivity assumption on $\mathbb{R}^{n}$ cannot be relaxed, as shown in \cite{K1}.

Later in \cite{DKP14,DKP16}, H\"{o}lder regularity and Harnack's inequality were established for fractional $p$-Laplacian type equations via the approach using nonlocal Caccioppoli estimates and logarithmic estimates. In particular, a new form of Harnack's inequality was obtained in \cite{DKP14} without an extra positivity assumption, which is an extension of the results in \cite{K2}. In \cite{C1}, the results in \cite{DKP14,DKP16} generalized to functions in a fractional De Giorgi class, regarding a wider class of nonlinear nonlocal problems with lower order terms. The approach used in \cite{C1} involves improved Caccioppoli estimates and nonlocal isoperimetric type inequalities. We refer to \cite{KMS15a,KMS15b,MSY,NoMAAN} for more on regularity results for fractional $p$-Laplacian type problems.

Subsequently, in \cite{GK1} the purely analytic approaches in \cite{DKP14,DKP16} were applied to mixed local and nonlocal $p$-Laplacian type equations.
For more on regularity and other qualitative behaviors, see \cite{BDVV21, BDVV22,BS,DM,GK0, SVWZ22}. In particular, the maximal regularity in \cite{DM} was achieved for general problems modeled after
\begin{align*}
v \mapsto \int_{\mathbb{R}^{n}}\int_{\mathbb{R}^{n}}\dfrac{|v(x)-v(y)|^{p}}{|x-y|^{n+sp}}\,dxdy+\int_{\Omega}|Dv|^{q}\,dx, \qquad sp<q.
\end{align*}

Such mixed local and nonlocal problems have an anisotropic feature, and they naturally link to other kinds of problems, namely nonlocal problems with nonstandard growth. Recently, the methods and results in \cite{DKP14,DKP16} have been extended to nonlocal problems with various nonstandard growth conditions \cite{BKO,BKS, BOS, DP19JDE, FSV22, FZ, MS, Ok1}; see also \cite{CK, CKW1,CKW2} for the extensions of the methods in \cite{C1}. Specifically, in \cite{BOS} local boundedness and H\"older continuity results were proved for the nonlocal double phase problem
\[ v \mapsto \iint_{\mathcal{C}_{\Omega}}\left(\frac{|v(x)-v(y)|^{p}}{|x-y|^{n+sp}}+a(x,y)\frac{|v(x)-v(y)|^{q}}{|x-y|^{n+sp}}\right)\,dxdy, \qquad 0<s\le t<1<p \le q. \]
They are analogous to those for the local double phase problem
\begin{equation*}
v \mapsto \int_{\Omega}(|Dv|^{p} + a(x)|Dv|^{q})\,dx, \qquad 1<p<q.
\end{equation*}
We refer to \cite{BCM,BCM18,CM15a,CM15b,CM16JFA,DM19} and references therein for comprehensive regularity results, including gradient regularity, for local double phase problems. 
Of particular interest to the present paper are the low regularity results in \cite{BCM, CM15a, CMM15}, which we summarize as follows:
\begin{itemize}
\item If $a(\cdot) \in L^{\infty}_{\loc}(\Omega)$ and
\begin{align*}
\begin{cases}
 p\le q \leq  \frac{np}{n-p} & \text{when }\ p<n,\\  
p\le q< \infty & \text{when }\ p= n,
\end{cases}
\end{align*}
then $u \in L^{\infty}_{\loc}(\Omega)$.
\item If $p \le n$, $u \in L^{\infty}_{\mathrm{loc}}(\Omega)$ and
\begin{align*}
a(\cdot) \in C^{0,\alpha}_{\loc}(\Omega), \quad q \leq p + \alpha,
\end{align*}
then $u \in C^{0,\gamma}_{\loc}(\Omega)$ for some $\gamma \in (0,1)$.
\item If $u$ is nonnegative in a ball $B_{9R} \subset \Omega$, and
\begin{equation*}
\left\{
\begin{aligned}
u \in L^{\infty}_{\mathrm{loc}}(\Omega),\,\, a(\cdot) \in C^{0,\alpha}_{\loc}(\Omega), \,\, q \leq p + \alpha \quad & \text{when }\; p\le n, \\
a(\cdot) \in C^{0,\alpha}_{\loc}(\Omega), \,\, \frac{q}{p} \le 1 + \frac{\alpha}{n} \quad & \text{when }\; p>n,
\end{aligned}
\right.
\end{equation*}
then 
\begin{equation*}
\sup_{B_{R}}u \le c\inf_{B_{R}}u
\end{equation*}
for some $c>0$.
\end{itemize}

In this paper, we prove the local boundedness, H\"older continuity and Harnack's inequality for minimizers of \eqref{eq:ftnal0} under the natural assumptions on $s,p,q$ and $a(\cdot)$ which are analogous to those for local double phase problems. To the best of our knowledge, each of our results are the first one that deals with the functional \eqref{eq:ftnal0} to make a systematic study on the regularity of its minimizers in the literature. Moreover, Harnack's inequality stated in Theorem \ref{thm:har} below is a new result even when $a(\cdot)$ is constant. One of main points in our problem is the interplay between local and nonlocal phenomena in the double phase structure, which gives rise to several difficulties in combining the local theory with the nonlocal theory. 
Indeed, the approaches in the present paper are different from those in \cite{BOS, DM}. 
On one hand, since the second term in \eqref{eq:ftnal0} is of a local nature, we could not derive the same form of the logarithmic estimate as in \cite[Lemma~5.1]{BOS}. On the other hand, since \eqref{eq:ftnal0} features purely nonlocal behaviors on the set $\{a(x)=0\}$, we could not compare \eqref{eq:ftnal0} with a local problem as in \cite{DM}.  
We thus develop a different method motivated from the ones in \cite{BBL22,BCM,C1}, whose crucial tools include the expansion of positivity results described in Lemmas~\ref{lem:key} and \ref{lem:grow} below. For fractional $p$-Laplacian type problems, analogous results are proved in \cite[Lemma~6.3]{C1}, but their proofs are not directly applicable to our double phase setting. 
To overcome this difficulty, we first make use of the local boundedness of minimizers and the H\"older continuity of $a(\cdot)$ to establish an improved Caccioppoli estimate given in Lemma~\ref{lem:DG} below. Then we prove Lemma~\ref{lem:key} by considering two alternatives, say ``the nonlocal phase'' and ``the mixed phase''. 
Moreover, for Harnack's inequality, we take into account the case that $sp > n$ as well, where optimal assumptions on $s,p,q$ and $a(\cdot)$ accordingly change. We also describe the precise nonlocal contribution of minimizers to our results via nonlocal tails.

\subsection{Assumptions and main results}
We actually consider a general functional of the type
\begin{align*}
\mathcal{E}(u;\Omega)\coloneqq\iint_{\mathcal{C}_{\Omega}}|u(x)-u(y)|^{p}K_{sp}(x,y)\,dxdy+\int_{\Omega}a(x)F(x,Du)\,dx,
\end{align*}
where $F:\Omega\times\ern\rightarrow\er$ is a Carath\'{e}odory function such that
\begin{align}\label{eq:F}
\Lambda^{-1}|\xi|^q\leq F(x,\xi)\leq \Lambda|\xi|^q
\end{align}
for some $\Lambda>1$, and $K_{sp}:\ern\times\ern\rightarrow\er$ is a symmetric kernel with order $(s,p)$; i.e., it is a measurable function satisfying 
\begin{align}\label{eq:ker}
\dfrac{\Lambda^{-1}}{|x-y|^{n+sp}}\leq K_{sp}(x,y) = K_{sp}(y,x) \leq\dfrac{\Lambda}{|x-y|^{n+sp}}
\end{align}
for a.e. $(x,y)\in\ern\times\ern$. 

Let us introduce relevant function spaces which will be used throughout the paper. We denote by
\begin{align}\label{eq:AOmega}
\mathcal{A}(\Omega)\coloneqq \left\{v:\ern\rightarrow\er\,\,\Big|\,\, v|_{\Omega}\in L^p(\Omega)\,\,\text{and}\,\,\mathcal{E}_0(v;\Omega)<\infty\right\},
\end{align}
and
\begin{align}\label{eq:tailsp}
L^{p-1}_{sp}(\ern)\coloneqq\left\{v:\ern\rightarrow\er\,\,\Big|\int_{\ern}\dfrac{|v(x)|^{p-1}}{(1+|x|)^{n+sp}}\,dx<\infty\right\}.
\end{align}
Then we define minimizers of the functional $\mathcal{E}$ as follows.
\begin{definition}
We say that $u\in\mathcal{A}(\Omega)$ is a minimizer of $\mathcal{E}$ if
\begin{align}\label{eq:mini}
\mathcal{E}(u;\Omega)\leq \mathcal{E}(v;\Omega)
\end{align}
for any measurable function $v:\ern\rightarrow\er$ with $v=u$ a.e. in $\ern\setminus\Omega$.
\end{definition}

The first theorem is about the local boundedness of minimizers in the case that $sp \le n$. 

\begin{theorem}[Local boundedness]\label{thm:bdd}
Assume \eqref{eq:spq}--\eqref{eq:ker} for the functional $\mathcal{E}$. Suppose that
$s,p,q$ satisfy 	
\begin{align}\label{eq:pq}
\begin{cases}
p\leq q\leq\frac{np}{n-sp}\quad&\text{when}\,\,\,sp<n,\\
p\leq q<\infty\quad&\text{when}\,\,\,sp = n.
\end{cases}
\end{align}
Then every minimizer $u\in\mathcal{A}(\Omega)\cap L^{p-1}_{sp}(\ern)$	of $\mathcal{E}$ is locally bounded in $\Omega$.
\end{theorem}

The second theorem is concerned with the H\"{o}lder continuity of bounded minimizers in the case that $sp\leq n$. 
For this, we additionally assume the H\"{o}lder continuity of $a(\cdot)$:
\begin{align}\label{eq:hol.a}
|a(x)-a(y)|\leq [a]_{\alpha}|x-y|^{\alpha},\quad\alpha\in(0,1]
\end{align}
for every $x,y\in\Omega$.

\begin{theorem}[H\"{o}lder continuity]\label{thm:hol}
Assume \eqref{eq:spq}, \eqref{eq:F} and \eqref{eq:ker} for the functional $\mathcal{E}$. Suppose that \eqref{eq:hol.a} holds for $a(\cdot):\Omega\rightarrow\er$. Let $s,p$ and $q$ satisfy $sp\leq n$ and
\begin{equation*}
q\leq sp+\alpha.
\end{equation*}	
Then every minimizer $u\in\mathcal{A}(\Omega)\cap L^{p-1}_{sp}(\ern)$	of $\mathcal{E}$ which is locally bounded in $\Omega$ is locally H\"{o}lder continuous in $\Omega$. Moreover, for any open set $\Omega'\Subset\Omega$, there exists $\gamma\in(0,1)$ depending only on $n,s,p,q,\Lambda,\alpha,[a]_{\alpha}$ and $\|u\|_{L^{\infty}(\Omega')}$ such that $u\in C^{\gamma}_{\loc}(\Omega')$.
\end{theorem}

Finally, in order to state a nonlocal version of Harnack's inequality, we define the tail as follows:
\begin{equation}\label{eq:tail}
\tail(v;x_0,R) \coloneqq \int_{\ern\setminus B_R(x_0)}\dfrac{|v(x)|^{p-1}}{|x-x_0|^{n+sp}}\,dx.
\end{equation}
For any subset $\Omega_0\Subset\Omega$, let us denote
\begin{align*}
\data(\Omega_0) \coloneqq
\begin{cases}
n,s,p,q,\Lambda,\alpha,[a]_{\alpha},\|u\|_{L^{\infty}(\Omega_0)}&\quad\text{when}\quad sp\leq n,\\
n,s,p,q,\Lambda,\alpha,[a]_{\alpha},[u]_{W^{s,p}(\Omega_0)}&\quad\text{when}\quad sp> n.
\end{cases}
\end{align*}
Then we have the following.
\begin{theorem}[Harnack's inequality]\label{thm:har}
Assume \eqref{eq:spq}, \eqref{eq:F} and \eqref{eq:ker} for the functional $\mathcal{E}$. Suppose that \eqref{eq:hol.a} for $a(\cdot)$. Let $s,p,q$ and $\alpha$ satisfy
\begin{align}\label{eq:qspa2}
\begin{cases}
q\leq sp+\alpha\quad&\text{when}\,\,\,sp\leq n,\\
q\leq p+\frac{p\alpha}{n}+\frac{(s-1)pq}{n}\quad&\text{when}\,\,\,sp> n.
\end{cases}
\end{align}
Let $u\in\mathcal{A}(\Omega)\cap L^{p-1}_{sp}(\ern)$ be a minimizer of $\mathcal{E}$ which is nonnegative in a ball $B_{16R}  = B_{16R}(x_{0}) \Subset \Omega$. When $sp \le n$, assume further that $u$ is bounded in $B_{16R}$. Then
\begin{align}\label{eq:har}
\sup_{B_{R}}u\leq c\inf_{B_{R}}u+c\left[R^{sp}\tail(u_-;x_0,2R)\right]^{\frac{1}{p-1}}
\end{align}	
holds for a constant $c = c(\data(B_{2R}))$, where $u_-=\max\{-u,0\}$.
\end{theorem}

\begin{remark}
We can obtain the following result for general minimizers by combining the results of Theorems~\ref{thm:bdd} and \ref{thm:hol}. Namely, under the same assumptions on $K_{sp}$, $F$ and $a(\cdot)$ as in Theorem~\ref{thm:hol}, every minimizer $u\in\mathcal{A}(\Omega)\cap L^{p-1}_{sp}(\ern)$ of $\mathcal{E}$ is locally H\"{o}lder continuous in $\Omega$, provided
\begin{align*}
\begin{cases}
p\leq q\leq\min\left\{\frac{np}{n-sp},sp+\alpha\right\}&\quad\text{when}\,\,\,sp<n,\\
p\leq q\leq sp+\alpha=n+\alpha&\quad\text{when}\,\,\,sp=n. 
\end{cases}
\end{align*}
Also, we can combine the results of Theorems~\ref{thm:bdd} and \ref{thm:har} as follows. With the same assumptions as Theorem \ref{thm:har} for $K_{sp}$, $F$ and $a(\cdot)$, if
\begin{align*}
\begin{cases}
p\leq q\leq\min\left\{\frac{np}{n-sp},sp+\alpha\right\}&\quad\text{when}\,\,\,sp<n,\\
p\leq q\leq sp+\alpha=n+\alpha&\quad\text{when}\,\,\,sp=n,\\
p\leq q\leq p+\frac{p\alpha}{n}+\frac{(s-1)pq}{n}&\quad\text{when}\,\,\,sp>n,
\end{cases}
\end{align*}
then we have the estimate \eqref{eq:har} for every minimizer $u\in\mathcal{A}(\Omega)\cap L^{p-1}_{sp}(\ern)$ of $\mathcal{E}$.
\end{remark}
We organize the paper as follows. Section~\ref{sec:2} is devoted to basic notations and inequalities which will be used throughout the paper. In Section~\ref{sec:3}, we obtain Caccioppoli estimates to prove Theorem~\ref{thm:bdd}. In Section~\ref{sec:4}, we prove the expansion of positivity lemma. Finally, in Section~\ref{sec:5} we prove Theorems~ \ref{thm:hol} and \ref{thm:har}.

\section{Preliminaries}\label{sec:2}
For $x_0\in\ern$ and $r>0$, $B_r(x_0)$ is the open ball in $\ern$ with center $x_0$ and radius $r$. We omit the center of a ball if it is not important in the context. Throughout the paper, $c$ is a general constant with $c\geq 1$, and its value may differ from each line. The notation $f\eqsim g$ means that there is a constant $c\geq 1$ such that $\frac{1}{c}f\leq g\leq cf$. We write $v_{\pm}\coloneqq \max\{\pm v,0\}$ for a measurable function $v$.
Additionally, if $v$ is integrable over a measurable set $S$ with $0<|S|<\infty$, we denote the integral average over $S$ by
\begin{align*}
(v)_S=\mean{S}v\,dx=\dfrac{1}{|S|}\int_{S}v\,dx.
\end{align*} 
When $S \subset \Omega$, we also denote
\[ a^{+}_{S} \coloneqq \sup_{x\in S} a(x) \qquad \text{and} \qquad a^{-}_{S} \coloneqq \inf_{x\in S}a(x). \]

We recall the definition and basic properties of fractional Sobolev spaces; for more details, see \cite{DPV}. 
With an open set $U\subseteq\ern$, $s\in(0,1)$ and $p\geq 1$, the fractional Sobolev space $W^{s,p}(U)$ consists of all measurable functions $v:U\rightarrow\er$ with
\begin{align*}
\|v\|_{W^{s,p}(U)} \coloneqq \|v\|_{L^p(U)}+[v]_{W^{s,p}(U)}=\left(\int_{U}|v|^p\,dx\right)^{\frac{1}{p}}+\left(\int_{U}\int_{U}\dfrac{|v(x)-v(y)|^p}{|x-y|^{n+sp}}\,dxdy\right)^{\frac{1}{p}}<\infty.
\end{align*}
We denote the $s$-fractional Sobolev conjugate of $p$ by
\begin{align*}
p^*_s=
\begin{cases}
\frac{np}{n-sp}&\text{if }sp<n,\\
\text{any number in }(p,\infty)&\text{if }sp\geq n.
\end{cases}
\end{align*}
Then we have the following embedding of $W^{s,p}(U)$, which holds, for instance, when $U$ is a Lipschitz domain (see for instance \cite{DPV}):
\begin{itemize}
\item If $sp \le n$, then $W^{s,p}(U) \hookrightarrow L^{p_{s}^{*}}(U)$.
\item If $sp > n$, then $W^{s,p}(U) \hookrightarrow C^{0,s-\frac{n}{p}}(U)$.
\end{itemize}
Moreover, we recall the corresponding fractional Sobolev-Poincar\'{e} type inequality on balls.
\begin{lemma}[\cite{NoMAAN,NoIHP}]
Let $s\in(0,1)$ and $p\geq 1$. For any $v\in W^{s,p}(B_r)$ there holds
\begin{align}\label{eq:fsp}
\left(\mean{B_r}|v-(v)_{B_r}|^{p^*_s}\,dx\right)^{\frac{p}{p^*_s}}\leq cr^{sp}\mean{B_r}\int_{B_r}\dfrac{|v(x)-v(y)|^p}{|x-y|^{n+sp}}\,dydx
\end{align}
for a constant $c=c(n,s,p)$. Moreover, if $sp > n$, then there holds
\begin{align}\label{eq:fsp2}
[v]_{C^{0,s-\frac{n}{p}}(B_r)}\leq c[v]_{W^{s,p}(B_r)}
\end{align}
for a constant $c=c(n,s,p)$.
\end{lemma}
Recalling \eqref{eq:ftnal0} and \eqref{eq:AOmega}, we see that
\begin{align*}
\mathcal{A}(\Omega)\subset W^{s,p}(\Omega).
\end{align*}
Also, \eqref{eq:fsp} implies that
\begin{align*}
\mathcal{A}(\Omega)\subset L^q(\Omega)\quad\text{if}\,\,
\begin{cases}
p \le q\leq\frac{np}{n-sp}\quad&\text{when}\,\, sp<n,\\
p \le q<\infty&\text{when}\,\,sp\geq n.
\end{cases}
\end{align*}

We next recall the tail space and tail given in \eqref{eq:tailsp} and \eqref{eq:tail}, respectively. Note that our definition of nonlocal tail is slightly different from those in \cite{C1,DKP14,DKP16}. 
Observe that if $v\in L^{q_0}(\ern)$ for some $q_0\geq p-1$, or if $v\in L^{p-1}(B_{R}(0))\cap L^{\infty}(\ern\setminus B_{R}(0))$ for some $R>0$, then $v\in L^{p-1}_{sp}(\ern)$. In particular, we have $W^{s,p}(\ern)\subset L^{p-1}_{sp}(\ern)$.
From the inequality
\begin{align*}
\dfrac{1+|x|}{|x-x_0|}\leq\dfrac{1+|x-x_0|+|x_0|}{|x-x_0|}\leq 1+\frac{1+|x_0|}{R} \qquad \text{for } x \in \mathbb{R}^{n}\setminus B_{R}(x_{0}),
\end{align*}
we have $\tail(v;x_0,R)<\infty$ for any $v\in L^{p-1}_{sp}(\ern)$ and $B_R(x_0)\subset\ern$.
If the center $x_0$ is not important, then we omit it and simply write $\tail(v;x_{0},R) \equiv \tail(v;R)$. 

\subsection{Useful lemmas}
We collect some inequalities which will be used in the proof of main theorems. The following lemma will be used in the proof of Theorem \ref{thm:bdd}; its proof is essentially the same as that of \cite[Lemma~2.4]{BOS}.

\begin{lemma}\label{lem:sp2}
Let the constants $s,p$ and $q$ satisfy \eqref{eq:spq} and \eqref{eq:pq}. Then for any $f\in W^{s,p}(B_r)$ and any constant $L_0\geq 0$, we have
\begin{equation*}
\begin{aligned}
\mean{B_r}\left(\left|\dfrac{f}{r^s}\right|^p+L_0\left|\dfrac{f}{r}\right|^q\right)\,dx&\leq cL_0r^{(s-1)q}\left(\mean{B_r}\int_{B_r}\dfrac{|f(x)-f(y)|^p}{|x-y|^{n+sp}}\,dxdy\right)^{\frac{q}{p}}\\
&\quad +c\left(\dfrac{|\supp f|}{|B_r|}\right)^{\frac{sp}{n}}\mean{B_r}\int_{B_r}\dfrac{|f(x)-f(y)|^p}{|x-y|^{n+sp}}\,dxdy\\
&\quad+c\left(\dfrac{|\supp f|}{|B_r|}\right)^{p-1}\mean{B_r}\left(\left|\dfrac{f}{r^s}\right|^p+L_0\left|\dfrac{f}{r}\right|^q\right)\,dx
\end{aligned}
\end{equation*}
for some constant $c=c(n,s,p,q)$ independent of $L_0$.
\end{lemma}

The following lemma is originated from \cite{D1} and used in \cite{BCM}.
\begin{lemma}\label{lem:iso}
Let $u\in W^{1,1}(B)$ for some ball $B\subset\ern$ and $m,l\in\er$ with $m<l$. Then we have
\begin{align*}
(l-m)|B\cap\{u\leq m\}|^{1-\frac{1}{n}}\leq \dfrac{c(n)|B|}{|B\cap\{u\geq l\}|}\int_{B\cap\{m<u\leq l\}}|Du|\,dx.
\end{align*}
\end{lemma}

Finally, we need the iteration lemma from \cite[Lemma 7.1]{Giu1}.
\begin{lemma}\label{lem:it}
Let $\{y_i\}^{\infty}_{i=0}$ be a sequence of nonnegative numbers with the inequality
\begin{align*}
y_{i+1}\leq b_1b_2^iy_i^{1+\beta},\quad i=0,1,2,\dots
\end{align*}
for some constants $b_1,\beta>0$ and $b_2>1$. If
\begin{align*}
y_0\leq b_1^{-1/\beta}b_2^{-1/\beta^2},
\end{align*}
then $y_i\rightarrow 0$ as $i\rightarrow \infty$.
\end{lemma}
\section{Caccioppoli estimates and local boundedness}\label{sec:3}
First, we show a Caccioppoli type estimate with tail, which plays an important role throughout the paper. 
\begin{lemma}\label{lem:cacc}
Let $u\in\mathcal{A}(\Omega)\cap L^{p-1}_{sp}(\Omega)$ be a minimizer of $\mathcal{E}$ under the assumptions \eqref{eq:spq}-\eqref{eq:ker}. Then for any ball $B_{2r}=B_{2r}(x_0)\Subset\Omega$ and $0<\rho<\sigma\leq r$ we have
\begin{equation*}
\begin{aligned}
\int_{B_{\rho}}&\int_{B_\rho}\dfrac{|w_{\pm}(x)-w_{\pm}(y)|^p}{|x-y|^{sp}}\dfrac{dxdy}{|x-y|^n}+\int_{B_{\rho}}a(x)|Dw_{\pm}|^q\,dx\\
&\quad+\int_{B_{\rho}}w_{\pm}(x)\left[\int_{\ern}\dfrac{w_{\mp}^{p-1}(y)}{|x-y|^{n+sp}}\,dy\right]\,dx\\
&\quad\quad\quad\leq \dfrac{c}{(\sigma-\rho)^{p}}\int_{B_{\sigma}}\int_{B_{\sigma}}\dfrac{|w_{\pm}(x)+w_{\pm}(y)|^p}{|x-y|^{(s-1)p}}\dfrac{dxdy}{|x-y|^n}+\dfrac{c}{(\sigma-\rho)^q}\int_{B_{\sigma}}a(x)\left|w_{\pm}\right|^q\,dx\\
&\quad\quad\quad\quad+\frac{c\sigma^{n+sp}}{(\sigma-\rho)^{n+sp}}
[\tail(w_{\pm};\sigma)]\int_{B_{\sigma}}w_{\pm}\,dx
\end{aligned}
\end{equation*}
for some $c=c(n,s,p,q,\Lambda)$, where $w_{\pm} \coloneqq (u-k)_{\pm}$  with $k\geq 0$.
\end{lemma}
\begin{proof}
We only prove the estimate for $w_+$, since the proof of the one for $w_-$ is similar. Choose two radii $\rho_{1},\sigma_{1}$ satisfying $\rho\leq\rho_1<\sigma_1\leq\sigma$ and then a cut-off function $\phi\in C^{\infty}_0(B_{\frac{\sigma_1+\rho_1}{2}})$ satisfying $0\leq\phi\leq 1$, $\phi=1$ on $B_{\rho_1}$ and $|D\phi|\leq\frac{4}{\sigma_1-\rho_1}$. We test \eqref{eq:mini} with $v=u-\phi w_+$. Then, since $u=v$ in $\ern\setminus B_{\sigma}$, we have
\begin{align}\label{eq:ca1.1}
\begin{split}
0&\leq \iint_{\mathcal{C}_{\Omega}}(|v(x)-v(y)|^{p}-|u(x)-u(y)|^{p})K_{sp}(x,y)\,dxdy\\
&\quad+\int_{\Omega}a(x)(F(x,Dv)-F(x,Du))\,dx\\
&\leq \int_{B_{\sigma}}\int_{B_{\sigma}}(|v(x)-v(y)|^{p}-|u(x)-u(y)|^{p})K_{sp}(x,y)\,dxdy\\
&\quad+2\int_{\ern\setminus B_{\sigma}}\int_{B_{\sigma}}(|v(x)-v(y)|^{p}-|u(x)-u(y)|^{p})K_{sp}(x,y)\,dxdy\\
&\quad+\int_{B_{\sigma}}a(x)(F(x,Dv)-F(x,Du))\,dx\\
&\eqqcolon I_1+I_2+I_3.
\end{split}
\end{align} 

Both $I_1$ and $I_2$ are estimated in the same way as in the proof of \cite[Proposition~7.5]{C1}:
\begin{align}\label{eq:ca2}
\begin{split}
I_1+I_2
&\leq c\int_{B_{\sigma_1}\setminus B_{\rho_1}}\int_{B_{\sigma_1}\setminus B_{\rho_1}}\dfrac{|w_+(x)-w_+(y)|^p}{|x-y|^{sp}}\dfrac{dxdy}{|x-y|^n}\\
&\quad+\dfrac{c}{(\sigma_1-\rho_1)^p}\int_{B_{\sigma}}\int_{B_{\sigma}}\dfrac{\left|w_+(x)+w_+(y)\right|^p}{|x-y|^{(s-1)p}}\dfrac{dxdy}{|x-y|^n}\\
&\quad-\frac{1}{c}\int_{B_{\rho_1}}\int_{B_{\rho_1}}\dfrac{|w_+(x)-w_+(y)|^p}{|x-y|^{sp}}\dfrac{dxdy}{|x-y|^n}\\
&\quad-\frac{1}{c}\int_{B_{\rho_1}}\int_{\ern}\dfrac{w_{-}(y)^{p-1}w_+(x)}{|x-y|^{sp}}\dfrac{dydx}{|x-y|^n}\\
&\quad +c\frac{\sigma^{n+sp}}{(\sigma_1-\rho_1)^{n+sp}}\left[\tail(w_{+};x_{0},\sigma)\right]\int_{B_{\sigma}}w_+\,dx
\end{split}
\end{align}
with $c=c(n,s,p,\Lambda)$. 

For $I_3$, we note that $\supp(u-v) \subset A^{+}(k,\sigma_{1}) \coloneqq \{x \in B_{\sigma_{1}}: u(x) \ge k \}$, which implies
\begin{align*}
\begin{split}
I_3&=\int_{A^{+}(k,\sigma_{1})}a(x)(F(x,Dv)-F(x,Du))\,dx\\
&\leq \Lambda\int_{A^{+}(k,\sigma_{1})}a(x)|Dv|^q\,dx-\Lambda^{-1}\int_{A^{+}(k,\sigma_1)}a(x)|Du|^q\,dx\\
&\leq \Lambda\int_{A^{+}(k,\sigma_1)}a(x)|Dv|^q\,dx-\Lambda^{-1}\int_{A^{+}(k,\sigma_1)}a(x)|Dw_+|^q\,dx.
\end{split}
\end{align*}
Here, we observe that
\begin{align*}
|Dv|^q & = |Du-(D\phi)w_+-\phi (Dw_+)|^q\\
& = |(1-\phi)Dw_+-(D\phi)w_+|^q \\
& \leq c|(1-\phi)Dw_+|^q+c\left|\dfrac{w_+}{\sigma_1-\rho_1}\right|^q
\end{align*}
holds in $A^{+}(k,\sigma_{1})$. 
In turn, we have
\begin{align}\label{eq:ca6}
\begin{split}
I_3&\leq c\int_{B_{\sigma_1}}a(x)\left(|(1-\phi)Dw_+|^q+\left|\dfrac{w_+}{\sigma_1-\rho_1}\right|^q\right)\,dx-\Lambda^{-1}\int_{B_{\sigma_1}}a(x)|Dw_+|^q\,dx\\
&\leq c\int_{B_{\sigma_1}\setminus B_{\rho_1}}a(x)|Dw_+|^q\,dx+c\int_{B_{\sigma}}a(x)\left|\dfrac{w_+}{\sigma_1-\rho_1}\right|^q\,dx-\Lambda^{-1}\int_{B_{\sigma_1}}a(x)|Dw_+|^q\,dx
\end{split}
\end{align}
for a constant $c = c(n,q,\Lambda)$. 

Combining the above estimates \eqref{eq:ca1.1}, \eqref{eq:ca2} and \eqref{eq:ca6}, we find
\begin{align}\label{eq:ca7}
\begin{split}
\int_{B_{\rho_1}}&\int_{B_{\rho_1}}\dfrac{|w_+(x)-w_+(y)|^p}{|x-y|^{sp}}\dfrac{dxdy}{|x-y|^n}+\int_{B_{\rho_1}}a(x)|Dw_+|^q\,dx\\
&+\int_{B_{\rho_1}}\int_{\ern}\dfrac{w_{-}(y)^{p-1}w_+(x)}{|x-y|^{sp}}\dfrac{dydx}{|x-y|^n}\\
&\quad\leq c\left(\int_{B_{\sigma_1}\setminus B_{\rho_1}}\int_{B_{\sigma_1}\setminus B_{\rho_1}}\dfrac{|w_+(x)-w_+(y)|^p}{|x-y|^{sp}}\dfrac{dxdy}{|x-y|^n}+\int_{B_{\sigma_1}\setminus B_{\rho_1}}a(x)|Dw_+|^q\,dx\right)\\
&\quad\quad+\dfrac{c}{(\sigma_1-\rho_1)^p}\int_{B_{\sigma}}\int_{B_{\sigma}}\dfrac{\left|w_+(x)+w_+(y)\right|^p}{|x-y|^{(s-1)p}}\dfrac{dxdy}{|x-y|^n}\\
&\quad\quad+\dfrac{c}{(\sigma_1-\rho_1)^q}\int_{B_{\sigma}}a(x)\left|w_+\right|^q\,dx\\
&\quad\quad+\frac{c\sigma^{n+sp}}{(\sigma_1-\rho_1)^{n+sp}}[\tail(w_{+};x_{0},\sigma)]\int_{B_{\sigma}}w_+\,dx.
\end{split}	
\end{align}
Now, we define for $t>0$,
\begin{align*}
\Phi(t)&=\int_{B_{t}}\int_{B_{t}}\dfrac{|w_+(x)-w_+(y)|^p}{|x-y|^{sp}}\dfrac{dxdy}{|x-y|^n}\\
&\quad+\int_{B_{t}}a(x)|Dw_+|^q\,dx+\int_{B_{t}}\int_{\ern}\dfrac{w_{-}(y)^{p-1}w_+(x)}{|x-y|^{sp}}\dfrac{dydx}{|x-y|^n}.
\end{align*}
Then \eqref{eq:ca7} reads as
\begin{align*}
\Phi(\rho_1)\leq c(\Phi(\sigma_1)-\Phi(\rho_1))&+\dfrac{c}{(\sigma_1-\rho_1)^p}\int_{B_{\sigma}}\int_{B_{\sigma}}\dfrac{\left|w_+(x)+w_+(y)\right|^p}{|x-y|^{(s-1)p}}\dfrac{dxdy}{|x-y|^n}\\
&+\dfrac{c}{(\sigma_1-\rho_1)^q}\int_{B_{\sigma}}a(x)\left|w_+\right|^q\,dx\\
&+\frac{c\sigma^{n+sp}}{(\sigma_1-\rho_1)^{n+sp}}[\tail(w_{+};x_{0},\sigma)]\int_{B_{\sigma}}w_+\,dx
\end{align*}
with $c=c(n,s,p,q,\Lambda)$. Now, the technical lemma~\cite[Lemma 2.5]{DM} gives the conclusion.
\end{proof}

We now prove the local boundedness result in Theorem \ref{thm:bdd}.

\begin{proof}[Proof of Theorem \ref{thm:bdd}] Throughout the proof, we denote
\begin{align*}
H_0(t) \coloneqq t^p+\|a\|_{L^{\infty}}t^q\qquad (t\geq 0).
\end{align*}	
Fix a ball $B_r\equiv B_r(x_0)\Subset\Omega$ with $r\leq 1$. Let $r/2\leq\rho<\sigma\leq r$ and $k>0$. We define the upper level set
\begin{align*}
A^+(k,\rho)\coloneqq\{x\in B_{\rho}:u(x)\geq k\}.
\end{align*}
Applying Lemma \ref{lem:sp2} with $f\equiv (u-k)_+$, we obtain 
\begin{align}\label{eq:bdd1}
\begin{split}
\rho^{-sp}\mean{B_{\rho}}H_0(f)\,dx&\leq\mean{B_{\rho}}\left[\left(\dfrac{f}{\rho^s}\right)^p+\|a\|_{L^{\infty}}\left(\dfrac{f}{\rho}\right)^q\right]\,dx\\
&\leq c\|a\|_{L^{\infty}}\rho^{(s-1)q}\left(\mean{B_{\rho}}\int_{B_{\rho}}\dfrac{|f(x)-f(y)|^p}{|x-y|^{n+sp}}\,dxdy\right)^{\frac{q}{p}}\\
&\quad+c\left(\dfrac{|A^+(k,\rho)|}{|B_{\rho}|}\right)^{\frac{sp}{n}}\mean{B_{\rho}}\int_{B_{\rho}}\dfrac{|f(x)-f(y)|^p}{|x-y|^{n+sp}}\,dxdy\\
&\quad+c\left(\dfrac{|A^+(k,\rho)|}{|B_{\rho}|}\right)^{p-1}\mean{B_{\sigma}}\left[\left(\dfrac{f}{\rho^s}\right)^p+\|a\|_{L^{\infty}}\left(\dfrac{f}{\rho}\right)^q\right]\,dx.
\end{split}
\end{align}
For fixed $0<h<k$, we see that 
\begin{align*}
(u(x)-h)_+=u(x)-h\geq k-h\quad\text{and}\quad
(u(x)-h)_+=u(x)-h\geq u(x)-k=(u(x)-k)_+
\end{align*}
for $x\in A^+(k,\rho)\subset A^+(h,\rho)$. 
Then we find
\begin{align*}
\begin{split}
\mean{B_{\rho}}(u-k)_+\,dx&\leq\mean{B_{\rho}}(u-h)_+\left(\dfrac{(u-h)_+}{k-h}\right)^{p-1}\,dx\leq\dfrac{1}{(k-h)^{p-1}}\mean{B_{\sigma}}H_0((u-h)_+)\,dx
\end{split}
\end{align*}
and
\begin{align}\label{eq:bdd3}
\begin{split}
\dfrac{|A^+(k,\rho)|}{|B_{\rho}|}&\leq\dfrac{1}{|B_{\rho}|}\int_{A^+(k,\rho)}\dfrac{(u-h)^p_+}{(k-h)^p}\,dx\\
&\leq\dfrac{1}{(k-h)^p|B_{\rho}|}\int_{A^+(h,\rho)}H_0((u-h)_+)\,dx\\
&\leq\dfrac{1}{(k-h)^p}\mean{B_{\rho}}H_0((u-h)_+)\,dx.
\end{split}
\end{align}
By Lemma \ref{lem:cacc}, we have
\begin{equation*}
\begin{aligned}
\lefteqn{ \mean{B_{\rho}}\int_{B_{\rho}}\dfrac{|f(x)-f(y)|^p}{|x-y|^{n+sp}}\,dxdy }\\
&\leq\dfrac{c}{(\sigma-\rho)^p}\mean{B_{\sigma}}(u(x)-h)^p_+\int_{B_{\sigma}}\dfrac{1}{|x-y|^{n+(s-1)p}}\,dydx+\dfrac{c\|a\|_{L^{\infty}}}{(\sigma-\rho)^q}\mean{B_{\sigma}}(u-h)^q_+\,dx\\
&\quad+c\left(\frac{\sigma^{n+sp}}{(\sigma-\rho)^{n+sp}}\tail(f;\sigma)\right)\mean{B_{\sigma}}(u-h)_+\,dx\\
&\leq \dfrac{c\rho^{(1-s)p}}{(\sigma-\rho)^p}\mean{B_{\sigma}}(u-h)^p_+\,dx+\dfrac{c\|a\|_{L^{\infty}}}{(\sigma-\rho)^q}\mean{B_{\sigma}}(u-h)^q_+\,dx\\
&\quad+\frac{c\sigma^{n+sp}}{(\sigma-\rho)^{n+sp}}(\tail(f;\sigma))\mean{B_{\sigma}}(u-h)_+\,dx\\
&\leq\dfrac{c}{(\sigma-\rho)^q}\mean{B_{\sigma}}H_0((u-h)_+)\,dx+\frac{c\tail(f;\sigma)}{(\sigma-\rho)^{n+sp}}\mean{B_{\sigma}}(u-h)_+\,dx.
\end{aligned}
\end{equation*}
Recalling $f\equiv (u-k)_+$ and combining the above estimate with \eqref{eq:bdd1}--\eqref{eq:bdd3} yield
\begin{equation*}
\begin{aligned}
\lefteqn{ \rho^{-sp}\mean{B_{\rho}}H_0((u-k)_+)\,dx }\\
&\leq \dfrac{c\rho^{(s-1)q}}{(\sigma-\rho)^{q^2/p}}\left(\mean{B_{\sigma}}H_0((u-h)_+)\,dx\right)^{\frac{q}{p}}\\
&\quad+\frac{c}{(k-h)^{q/p'}}\dfrac{\rho^{(s-1)q}[\tail((u-k)_+;\sigma)]^{q/p}}{(\sigma-\rho)^{(n+sp)q/p}}\left(\mean{B_{\sigma}}H_0((u-h)_+)\,dx\right)^{\frac{q}{p}}\\
&\quad+\dfrac{c}{(k-h)^{sp^2/n}}\dfrac{1}{(\sigma-\rho)^q}\left(\mean{B_{\sigma}}H_0((u-h)_+)\,dx\right)^{1+\frac{sp}{n}}\\
&\quad+\dfrac{c\tail((u-k)_+;\sigma)}{(k-h)^{sp^2/n+p-1}(\sigma-\rho)^{n+sp}}\left(\mean{B_{\sigma}}H_0((u-k)_+)\,dx\right)^{1+\frac{sp}{n}}\\
&\quad+\dfrac{cr^{-q}}{(k-h)^{p(p-1)}}\left(\mean{B_{\sigma}}H_0((u-h)_+)\,dx\right)^p.
\end{aligned}
\end{equation*}

Now, for $i=0,1,2,\dots$ and $k_0>1$, we denote
\begin{align*}
\sigma_i \coloneqq \frac{r}{2}(1+2^{-i}),\quad k_i \coloneqq 2k_0(1-2^{-i-1})\quad\text{and}\quad y_i \coloneqq \int_{A^+(k_i,\sigma_i)}H_0((u-k_i)_+)\,dx.
\end{align*}
Since $H_0(u)\in L^1(\Omega)$ from \eqref{eq:fsp} and \eqref{eq:pq}, it follows that
\begin{align*}
y_0=\int_{A^+(k_0,r)}H_0((u-k_0)_+)\,dx\quad\longrightarrow\quad 0\quad\text{as}\quad k_0\rightarrow\infty.
\end{align*}
Consider a large number $k_0>1$ for which
\begin{align*}
y_i\leq y_{i-1}\leq\cdots\leq y_0\leq 1,\quad i=1,2,\dots.
\end{align*}
Then since $u\in L^{p-1}_{sp}(\ern)$ and so $\tail((u-k_i)_+;\sigma_i)\leq \tail(u;r/2)<\infty$, we obtain
\begin{equation*}
\begin{aligned}
y_{i+1}&\leq\tilde{c}\left(2^{iq^2/p}y_i^{q/p}+2^{i(q/p'+(n+sp)q/p)}y_i^{q/p}\right.\\
&\hspace{1cm}\left.+2^{i(sp^2/n+q)}y_i^{1+(sp/n)}+2^{i(sp^2/n+p+n+sp)}y_i^{1+(sp/n)}+2^{ip(p-1)}y_i^p\right)\\
&\leq \tilde{c}2^{\theta i}y_i^{1+\beta}
\end{aligned}
\end{equation*}
for some constant $\tilde{c}>0$ depending on $n,s,p,q,\Lambda,\|a\|_{L^{\infty}},r$ and $\tail(u;r/2)$, where
\begin{align*}
\theta=\max\left\{\frac{q^2}{p},\frac{q}{p'}+(n+sp)\frac{q}{p},\frac{sp^2}{n}+q,\frac{sp^2}{n}+p+n+sp,p(p-1)\right\}
\end{align*}
and
\begin{align*}
\quad\beta=\min\left\{\frac{q}{p}-1,\frac{sp}{n},p-1\right\}.
\end{align*}

Now we select a constant $k_0$ sufficiently large to satisfy
\begin{align*}
y_0\leq\tilde{c}^{-1/\beta}2^{-\theta/\beta^2}.
\end{align*}
Then Lemma \ref{lem:it} yields
\begin{align*}
y_{\infty}=\int_{A^+(2k_0,r/2)}H_0((u-2k_0)_+)\,dx=0,
\end{align*}
and so $u\leq k_0$ a.e. in $B_{r/2}$.

Applying the same argument to $-u$, we finally obtain $u\in L^{\infty}(B_{r/2})$. 
\end{proof}

\section{Expansion of positivity}\label{sec:4}
Throughout this section we assume that $K_{sp}$ is symmetric and satisfies \eqref{eq:ker}. We suppose \eqref{eq:F} for $F$, and let $a(\cdot)$ satisfy \eqref{eq:a} and \eqref{eq:hol.a}. Also, assume that $s,p,q$ and $\alpha$ satisfy \eqref{eq:spq} and \eqref{eq:qspa2}.

\begin{lemma}\label{lem:DG}
Let $u \in \mathcal{A}(\Omega) \cap L^{p-1}_{sp}(\mathbb{R}^{n})$ be a minimzer of $\mathcal{E}$, and let $B_{R} \Subset \Omega$ be a ball with $R \le 1$. When $sp \le n$, assume further that $u$ is bounded in $B_{R}$. Then for $w_{\pm}\coloneqq(u-k)_{\pm}$ with $|k| \le \|u\|_{L^{\infty}(B_{R})}$, we have
\begin{align}\label{eq:D}
\begin{split}
\lefteqn{ [w_{\pm}]^p_{W^{s,p}(B_{\rho})}+a^{-}_{B_{R}}[w_{\pm}]^q_{W^{1,q}(B_{\rho})}+\int_{B_{\rho}}w_{\pm}(x)\left[\int_{\ern}\dfrac{w_{\mp}^{p-1}(y)}{|x-y|^{n+sp}}\,dy\right]\,dx }\\
&\leq c\left(\frac{r}{r-\rho}\right)^{n+q}\left[ \dfrac{\|w_{\pm}\|^p_{L^p(B_r)}}{r^{sp}}+ a^{-}_{B_R}\dfrac{\|w_{\pm}\|^q_{L^q(B_r)}}{r^q} 
+ \|w_{\pm}\|_{L^1(B_r)}\tail(w_{\pm};r)\right]
\end{split}
\end{align}
for a constant $c=c(\data(B_{R}))$, whenever $B_{\rho} \subset B_{r} \subset B_{R}$ are concentric balls with $R/2 \le \rho \le r \le R$.
\end{lemma}
\begin{proof}
Lemma~\ref{lem:cacc} directly implies
\begin{align}\label{eq:DG1}
\begin{split}
\lefteqn{ [w_{\pm}]^p_{W^{s,p}(B_{\rho})}+a^{-}_{B_R}[w_{\pm}]^q_{W^{1,q}(B_{\rho})}+\int_{B_{\rho}}w_{\pm}(x)\left[\int_{\ern}\dfrac{w_{\mp}^{p-1}(y)}{|x-y|^{n+sp}}\,dy\right]\,dx }\\
&\leq \dfrac{c}{(r-\rho)^{p}}\int_{B_{r}}\int_{B_{r}}\dfrac{|w_{\pm}(x)+w_{\pm}(y)|^p}{|x-y|^{(s-1)p}}\dfrac{dxdy}{|x-y|^n}+\dfrac{c}{(r-\rho)^q}\int_{B_{r}}a(x)|w_{\pm}|^q\,dx\\
&\quad+\dfrac{cr^{n+q}}{(r-\rho)^{n+q}}\|w_{\pm}\|_{L^1(B_r)}\tail(w_{\pm};r).
\end{split}
\end{align}
We estimate the first integral in the right-hand side by using the symmetry of $x$ and $y$: 
\begin{align*}
\begin{split}
\lefteqn{ \dfrac{c}{(r-\rho)^{p}}\int_{B_{r}}\int_{B_{r}}\dfrac{|w_{\pm}(x)+w_{\pm}(y)|^p}{|x-y|^{(s-1)p}}\dfrac{dxdy}{|x-y|^n} }\\
&\leq \dfrac{c}{(r-\rho)^p}\int_{B_r}\int_{B_r}\dfrac{|w_{\pm}(x)|^p}{|x-y|^{(s-1)p}}\dfrac{dxdy}{|x-y|^n}\\
&\leq \dfrac{c}{(r-\rho)^p}\int_{B_r}|w_{\pm}(x)|^p\int_{B_{2r}(x)}\dfrac{1}{|x-y|^{n+(s-1)p}}\,\,dydx\\
&\leq \dfrac{cr^{(1-s)p}}{(r-\rho)^p}\|w_{\pm}\|^p_{L^p(B_r)} = c\left(\frac{r}{r-\rho}\right)^{p}\frac{\|w_{\pm}\|_{L^{p}(B_{r})}^{p}}{r^{sp}}.
\end{split}
\end{align*}
For the second one, we use \eqref{eq:hol.a} and the fact that $R/2 \le r \le R$ to have
\begin{align}\label{eq:DG3} 
\frac{1}{(r-\rho)^{q}}\int_{B_{r}}a(x)|w_{\pm}|^{q}\,dx 
& \le \frac{c}{(r-\rho)^{q}}\int_{B_{r}}a^{-}_{B_{R}}|w_{\pm}|^{q}\,dx + \frac{c}{(r-\rho)^{q}}\int_{B_{r}}r^{\alpha}|w_{\pm}|^{q}\,dx.
\end{align}
Here, when $sp\le n$, we use $\eqref{eq:qspa2}_{1}$ in order to estimate
\begin{align}\label{eq:DG4}
\frac{c}{(r-\rho)^{q}}\int_{B_{r}}r^{\alpha}|w_{\pm}|^{q}\,dx 
 \le \frac{c\|w_{\pm}\|_{L^{\infty}(B_{r})}^{q-p}}{(r-\rho)^{q}}\int_{B_{r}}r^{\alpha}|w_{\pm}|^{p}\,dx  \le c\left(\frac{r}{r-\rho}\right)^{q}\frac{\|w_{\pm}\|_{L^{p}(B_{r})}^{p}}{r^{sp}}
\end{align}
with $c=c(\data(B_{r}))$. When $sp>n$, note that $\eqref{eq:qspa2}_{2}$ is equivalent to
\begin{align*}
q \le sp+\alpha + \left(s-\frac{n}{p}\right)(q-p),
\end{align*}
which together with \eqref{eq:fsp2} implies
\begin{align*}
\left(\osc_{B_r}u\right)^{q-p}=\left(\frac{\osc_{B_r}u}{r^{s-\frac{n}{p}}}\right)^{q-p}r^{\left(s-\frac{n}{p}\right)(q-p)}\leq c[u]_{W^{s,p}(B_r)}^{q-p}r^{q-sp-\alpha}. 
\end{align*}
In turn, we obtain
\begin{align}\label{eq:DG*2}
\frac{c}{(r-\rho)^{q}}\int_{B_{r}}r^{\alpha}|w_{\pm}|^{q}\,dx \le \frac{c(\osc_{B_{r}}u)^{q-p}}{(r-\rho)^{q}}\int_{B_{r}}r^{\alpha}|w_{\pm}|^{p}\,dx  \le c\left(\frac{r}{r-\rho}\right)^{q}\frac{\|w_{\pm}\|_{L^{p}(B_{r})}^{p}}{r^{sp}}
\end{align}
for a constant $c=c(\data(B_{r}))$. Combining \eqref{eq:DG1}--\eqref{eq:DG3} with each of \eqref{eq:DG4} and \eqref{eq:DG*2}, and then recalling the fact $p\leq q$, in any case we conclude with \eqref{eq:D}.
\end{proof}

In the following, with $B_{r} \Subset \Omega$ being any ball, we denote
\begin{equation}\label{eq:Gg}
 G_{B_{r}}(t) \coloneqq \frac{t^{p}}{r^{sp}} + a^{-}_{B_{r}}\frac{t^{q}}{r^{q}} \quad \text{and} \quad g_{B_{r}}(t) \coloneqq \frac{t^{p-1}}{r^{sp}} + a^{-}_{B_{r}}\frac{t^{q-1}}{r^{q}} \qquad \text{for } t \ge 0.
\end{equation}
Now we prove the following key lemma. 

\begin{lemma}\label{lem:key}
Let $u \in \mathcal{A}(\Omega) \cap L^{p-1}_{sp}(\mathbb{R}^{n})$ be a minimizer of $\mathcal{E}$ which is nonnegative in a ball $B_{4R}\Subset \Omega$ with $R\leq 1$. When $sp \le n$, assume further that $u$ is bounded in $B_{4R}$. Suppose that
\begin{align}\label{eq:ke1}
|B_{2R}\cap\{u\geq t\}|\geq\nu|B_{2R}|
\end{align}
for some $\nu\in(0,1)$ and $t>0$. Then for any $\delta\in(0,\frac{1}{8}]$, if
\begin{align}\label{eq:ke2}
\tail(u_-;4R)\leq g_{B_{4R}}(\delta t),
\end{align}
then
\begin{align}\label{eq:ke3}
|B_{2R}\cap\{u< 2\delta t\}|\leq\dfrac{c_1}{\nu^{\max\{\frac{2q}{q-1},\frac{2n}{n-1}\}}}\left(\delta^{\frac{p-1}{2}}+\frac{1}{|\log\delta|^{\frac{n}{n-1}\frac{q-1}{q}}}\right)|B_{2R}|,
\end{align}
where $c_1=c_1(\data(B_{4R}))$.
\end{lemma}
\begin{proof}
We may assume that all the balls are centered at the origin.
We first observe that, for any $k\geq 0$ and $\zeta\geq 1$,
\begin{align}\label{eq:ke4}
\|(u-k)_-\|^{\zeta}_{L^{\zeta}(B_{4R})}=\int_{A^-(k,4R)}(k-u(x))^{\zeta}\,dx\leq |A^-(k,4R)|k^{\zeta}\leq |B_{4R}|k^{\zeta}.
\end{align}
Fix any $l\geq \frac{\delta t}{2}$. We apply \eqref{eq:ke2} to obtain
\begin{align}\label{eq:ke5}
\begin{split}
\tail((u-l)_-;4R)&=\int_{\ern\setminus B_{4R}}\dfrac{(l-u(x))^{p-1}_+}{|x|^{n+sp}}\,dx\\
&\leq c\left[l^{p-1}\int_{\ern\setminus B_{4R}}\dfrac{dx}{|x|^{n+sp}}+\int_{\ern\setminus B_{4R}}\dfrac{u_-(x)^{p-1}}{|x|^{n+sp}}\,dx\right]\\
&=c\left[R^{-sp}l^{p-1}+\tail(u_-;4R)\right]\\
&\leq cg_{B_{4R}}(l).
\end{split}
\end{align}
Then by Lemma \ref{lem:DG}, \eqref{eq:ke4} and \eqref{eq:ke5}, it follows that 
\begin{align}\label{eq:ke6}
\begin{split}
\lefteqn{ \int_{B_{2R}}\int_{B_{2R}}\dfrac{(u(x)-l)^{p-1}_{+}(u(y)-l)_{-}}{|x-y|^{n+sp}}\,dxdy+a^{-}_{B_{4R}}\int_{B_{2R}}|D(u-l)_-|^q\,dx }\\
& \leq c\left[\frac{\|(u-l)_{-}\|_{L^{p}(B_{4R})}^{p}}{R^{sp}} + a^{-}_{B_{4R}}\frac{\|(u-l)_{-}\|_{L^{q}(B_{4R})}^{q}}{R^{q}} + \|(u-l)_{-}\|_{L^{1}(B_{4R})}\tail((u-l)_{-};4R)\right] \\
& \leq cG_{B_{4R}}(l)|B_{R}|
\end{split}
\end{align}
for any $l\geq\frac{\delta t}{2}$ with $c=c(\data(B_{4R}))$. Here, if $a^{-}_{B_{4R}}=0$, then we have
\begin{align*}
\int_{B_{2R}}\int_{B_{2R}}\dfrac{(u(x)-l)^{p-1}_+(u(y)-l)_-}{|x-y|^{n+sp}}\,dxdy\leq c_1\left(\dfrac{l}{R^{s}}\right)^p|B_R|;
\end{align*}
in this case, \eqref{eq:ke3} follows in the same way as in \cite[Lemmas 6.3 and 6.5]{C1}. 

We now consider the case $a^{-}_{B_{4R}}>0$. Note that we have $u\in W^{1,q}(B_{4R})$ in this case.
We consider the following two cases:
\begin{align}\label{eq:alt}
\left(\dfrac{\delta t}{R^s}\right)^p\geq a^{-}_{B_{4R}}\left(\dfrac{\delta t}{R}\right)^q\quad\text{and}\quad\left(\dfrac{\delta t}{R^s}\right)^p<a^{-}_{B_{4R}}\left(\dfrac{\delta t}{R}\right)^q.
\end{align}

\textit{Step 1: The case $\eqref{eq:alt}_{1}$.} 
In this case, by \eqref{eq:ke1} and \eqref{eq:ke6} with $l=4\delta t$, we have
\begin{equation*}
\begin{aligned}
\left(\dfrac{4\delta t}{R^s}\right)^p|B_{2R}|\overset{\eqref{eq:ke6}}&{\geq} \frac{1}{c}\int_{B_{2R}}\int_{B_{2R}}\dfrac{(u(x)-4\delta t)^{p-1}_{+}(u(y)-4\delta t)_{-}}{|x-y|^{n+sp}}\,dxdy\\
&\geq \dfrac{1}{cR^{n+sp}}\int_{B_{2R}\cap\{u \geq t\}}(u(x)-4\delta t)^{p-1}\,dx\int_{B_{2R}\cap\{u < 2\delta t\}}(4\delta t-u(y))\,dy\\
&\geq \dfrac{\delta t^p}{cR^{n+sp}}|B_{2R}\cap\{u\geq t\}||B_{2R}\cap\{u< 2\delta t\}|\\
\overset{\eqref{eq:ke1}}&{\geq} \dfrac{\delta\nu t^p}{cR^{sp}}|B_{2R}\cap\{u< 2\delta t\}|.
\end{aligned}
\end{equation*}
Then it follows that
\begin{align}\label{eq:ke7}
|B_{2R}\cap\{u< 2\delta t\}|\leq \dfrac{c\delta ^{p-1}}{\nu}|B_{2R}|
\end{align}
with $c=c(\data(B_{4R}))$.

We start to deal with the case $\eqref{eq:alt}_{2}$. We choose $i\in\mathbb{N}$ satisfying
\begin{align}\label{eq:ke7.1}
2^{-i-1}\leq 2\delta<2^{-i}
\end{align}
and further distinguish two subcases:
\begin{align}\label{eq:alt2}
\left(\dfrac{t}{R^s}\right)^p<a^{-}_{B_{4R}}\left(\dfrac{t}{R}\right)^q\quad\text{and}\quad\left(\dfrac{t}{R^s}\right)^p\geq a^{-}_{B_{4R}}\left(\dfrac{t}{R}\right)^q.
\end{align}

\textit{Step 2: The case $\eqref{eq:alt}_{2}$ and $\eqref{eq:alt2}_{1}$.}
Note from $\eqref{eq:alt}_{2}$ that
\begin{align*}
\left(\dfrac{2^{-i}t}{R^s}\right)^p=2^p\left(\dfrac{2^{-i-1}t}{R^s}\right)^p\leq 2^p\left(\dfrac{2\delta t}{R^s}\right)^p\leq 2^pa^{-}_{B_{4R}}\left(\dfrac{2\delta t}{R}\right)^q\leq 2^pa^{-}_{B_{4R}}\left(\dfrac{2^{-i}t}{R}\right)^q.
\end{align*}
Then together with $\eqref{eq:alt2}_{1}$ we have
\begin{align}\label{eq:ke8}
\left(\dfrac{kt}{R^s}\right)^{p}\leq 2^pa^{-}_{B_{4R}}\left(\dfrac{kt}{R}\right)^q\quad\text{for all }k\in[2^{-i},1].
\end{align}
Denoting $A_j \coloneqq B_{2R}\cap\{2^{-j}t<u\leq 2^{-j+1}t\}$, it follows from Lemma \ref{lem:iso} that for all $j\in\{1,\dots,i\}$,
\begin{align}\label{eq:ke8.1}
\begin{split}
2^{-j}t|B_{2R}\cap\{u\leq 2^{-j}t\}|^{1-\frac{1}{n}}&\leq \dfrac{c|B_{2R}|}{|B_{2R}\cap\{u\geq 2^{-j+1}t\}|}\int_{A_j}|Du|\,dx\\
&\leq\dfrac{c}{\nu}\int_{A_j}|Du|\,dx,
\end{split}
\end{align}
where for the last inequality we have used
\begin{align*}
|B_{2R}\cap\{u\geq 2^{-j+1}t\}|\geq|B_{2R}\cap\{u\geq t\}|\geq \nu|B_{2R}|.
\end{align*}
Then it follows that
\begin{align*}
\nu(2^{-j}t)|B_{2R}\cap\{u\leq 2^{-j}t\}|^{\frac{n-1}{n}}\leq c\int_{A_j}|Du|\,dx.
\end{align*}
Moreover, by H\"{o}lder's inequality, we find that
\begin{align*}
\int_{A_j}|Du|\,dx\leq|A_j|^{\frac{1}{q'}}\left(\int_{B_{2R}}|D(u-2^{-j+1}t)_-|^q\,dx\right)^{\frac{1}{q}}.
\end{align*}
Combining the last two displays, we have
\begin{equation*}
\begin{aligned}
a^{-}_{B_{4R}}\nu^q(2^{-j}t)^q|B_{2R}\cap\{u\leq 2^{-j}t\}|^{\frac{n-1}{n}q}&\leq c|A_{j}|^{q-1} a^{-}_{B_{4R}}\int_{B_{2R}}|D(u-2^{-j+1}t)_-|^q\,dx\\
\overset{\eqref{eq:ke6}}&{\leq} c|A_j|^{q-1}G_{B_{4R}}(2^{-j+1}t)|B_{2R}|\\
\overset{\eqref{eq:ke8}}&{\leq} c|A_j|^{q-1}a^{-}_{B_{4R}}\left(\dfrac{2^{-j+1}t}{R}\right)^q|B_{2R}|,
\end{aligned}
\end{equation*}
and so
\begin{align*}
\nu^{\frac{q}{q-1}}|B_{2R}\cap\{u\leq 2^{-j}t\}|^{\frac{n-1}{n}\frac{q}{q-1}}\leq c\left(\frac{|B_{2R}|}{R^q}\right)^{\frac{1}{q-1}}|A_j|\leq c R^{\frac{n-q}{q-1}}|A_j|.
\end{align*}
Then since $(B_{2R}\cap\{u\leq 2^{-i}t\})\subset (B_{2R}\cap\{u\leq 2^{-j}t\})$ for all $j\in\{1,\dots,i\}$, there holds
\begin{align}\label{eq:ke8.2}
\nu^{\frac{q}{q-1}}|B_{2R}\cap\{u\leq 2^{-i}t\}|^{\frac{n-1}{n}\frac{q}{q-1}}\leq c R^{\frac{n-q}{q-1}}|A_j|.
\end{align}
We recall the definition of $A_j$ and sum up \eqref{eq:ke8.2} over $j\in\{1,\dots,i\}$, to discover
\begin{equation*}
\begin{aligned}
i\nu^{\frac{q}{q-1}}|B_{2R}\cap\{u\leq 2^{-i}t\}|^{\frac{n-1}{n}\frac{q}{q-1}}&\leq c R^{\frac{n-q}{q-1}}\sum^{i}_{j=1}|A_j|\\
&\leq c R^{\frac{n-q}{q-1}}|B_{2R}\cap\{2^{-i}t< u\leq 2t\}|\\
&\leq c R^{\frac{n-q}{q-1}}|B_{2R}|\eqsim|B_{2R}|^{\frac{n-1}{n}\frac{q}{q-1}},
\end{aligned}
\end{equation*}
and so
\begin{align*}
|B_{2R}\cap\{u\leq 2^{-i}t\}|\leq \dfrac{c|B_{2R}|}{\nu^{\frac{n}{n-1}}i^{\frac{n}{n-1}\frac{q-1}{q}}}.
\end{align*}
Recalling \eqref{eq:ke7.1}, one can easily conclude that
\begin{align}\label{eq:ke9}
|B_{2R}\cap\{u\leq 2\delta t\}|\leq \dfrac{c|B_{2R}|}{\nu^{\frac{n}{n-1}}|\log\delta|^{\frac{n}{n-1}\frac{q-1}{q}}}
\end{align}
with $c=c(\data(B_{4R}))$.

\textit{Step 3: The case $\eqref{eq:alt}_{2}$ and $\eqref{eq:alt2}_{2}$.}
In this case, let $\beta\in\{1,\dots,i\}$ be such that
\begin{align*}
\left(\dfrac{2^{-\beta}t}{R^s}\right)^p<a^{-}_{B_{4R}}\left(\dfrac{2^{-\beta}t}{R}\right)^q\quad\text{but}\quad\left(\dfrac{2^{-\beta+1}t}{R^s}\right)^p\geq a^{-}_{B_{4R}}\left(\dfrac{2^{-\beta+1}t}{R}\right)^q.
\end{align*}
By using \eqref{eq:ke6} for $l=2^{-\beta+2}t$, we have
\begin{equation*}
\begin{aligned}
\lefteqn{ \left(\dfrac{2^{-\beta+2}t}{R^s}\right)^p|B_{2R}|\overset{\eqref{eq:ke6}}{\geq }\frac{1}{c}\int_{B_{2R}}\int_{B_{2R}}\dfrac{(u(x)-2^{-\beta+2}t)_+^{p-1}(u(x)-2^{-\beta+2}t)_-}{|x-y|^{n+sp}}\,dxdy }\\
&\geq  \dfrac{1}{cR^{n+sp}}\int_{B_{2R}\cap\{u\geq t\}}(u(x)-2^{-\beta+2}t)^{p-1}\,dx\int_{B_{2R}\cap\{u\leq 2^{-\beta+1}t\}}(2^{-\beta+2}t-u(y))\,dy\\
&\geq\dfrac{}{cR^{n+sp}}2^{-\beta+1}t^p|B_{2R}\cap\{u\geq t\}||B_{2R}\cap\{u\leq 2^{-\beta+1}t\}|\\
\overset{\eqref{eq:ke1}}&{\geq}\dfrac{1}{cR^{sp}}2^{-\beta+1}\nu t^p|B_{2R}\cap\{u\leq 2^{-\beta+1}t\}|.
\end{aligned}
\end{equation*}
Consequently,
\begin{align*}
|B_{2R}\cap\{u\leq 2^{-\beta+1}t\}|\leq\dfrac{c(2^{-\beta+1})^{p-1}}{\nu}|B_{2R}|.
\end{align*}
On the other hand, by the same computations as in \eqref{eq:ke8.1}, we discover that for all $j\in\{\beta,\dots,i\}$,
\begin{align*}
2^{-\beta}t|B_{2R}\cap\{u\leq 2^{-\beta}t\}|^{1-\frac{1}{n}}\leq\frac{c}{\nu}\int_{A_j}|Du|\,dx.
\end{align*}
Following the same arguments used in \eqref{eq:ke8.1}--\eqref{eq:ke8.2}, we have the inequality
\begin{align*}
\nu^{\frac{q}{q-1}}|B_{2R}\cap\{u\leq 2^{-i}t\}|^{\frac{n-1}{n}\frac{q}{q-1}}\leq c R^{\frac{n-q}{q-1}}|A_j|.
\end{align*}
Sum $j=\beta, \beta+1,\dots,i$, to find
\begin{equation*}
\begin{aligned}
\nu^{\frac{q}{q-1}}(i-\beta+1)|B_{2R}\cap\{u\leq 2^{-i}t\}|^{\frac{n-1}{n}\frac{q}{q-1}}&\leq c R^{\frac{n-q}{q-1}}\sum^i_{j=\beta}|A_j|\\
&\leq c R^{\frac{n-q}{q-1}}|B_{2R}\cap\{2^{-i}t< u\leq 2^{-\beta+1}t\}|\\
&\leq c R^{\frac{n-q}{q-1}}\dfrac{(2^{-\beta+1})^{p-1}}{\nu}|B_{2R}|\eqsim \dfrac{2^{-\beta(p-1)}}{\nu}|B_{2R}|^{\frac{n-1}{n}\frac{q}{q-1}}.
\end{aligned}
\end{equation*}
Thus
\begin{align*}
|B_{2R}\cap\{u\leq 2^{-i}t\}|\leq c|B_{2R}|\left(\dfrac{2^{-\beta(p-1)}}{\nu^{1+\frac{q}{q-1}}(i-\beta+1)}\right)^{\frac{n}{n-1}\frac{q-1}{q}},
\end{align*}
which together with \eqref{eq:ke7.1} implies
\begin{align*}
|B_{2R}\cap\{u\leq 2\delta t\}|\leq c |B_{2R}|\left(\dfrac{2^{-\beta(p-1)}}{\nu^{1+\frac{q}{q-1}}\left(\frac{-\log\delta}{\log 2}-\beta+1\right)}\right)^{\frac{n}{n-1}\frac{q-1}{q}}.
\end{align*}

We first consider the case that $2^{-\beta}\leq\delta^{\frac{1}{2}}$. Recall \eqref{eq:ke7.1} and the fact that $\beta\leq i-1$, to observe
\begin{align*}
\frac{-\log\delta}{\log 2}\geq\beta+2.
\end{align*}
Then we have
\begin{align}\label{eq:ke10}
|B_{2R}\cap\{u\leq 2\delta t\}|\leq c|B_{2R}|\left(\dfrac{\delta^{\frac{p-1}{2}}}{\nu^{1+\frac{q}{q-1}}}\right)^{\frac{n}{n-1}\frac{q-1}{q}}\leq c|B_{2R}|\left(\dfrac{\delta^{\frac{p-1}{2}}}{\nu^{\frac{2q}{q-1}}}\right)^{\frac{n}{n-1}\frac{q-1}{q}}.
\end{align}
We next consider the case $2^{-\beta}>\delta^{\frac{1}{2}}$. Then $\frac{-\log\delta}{2\log 2}>\beta$, and so
\begin{align}\label{eq:ke11}
\begin{split}
|B_{2R}\cap\{u\leq 2\delta t\}|&\leq c|B_{2R}|\left(\dfrac{1}{\nu^{\frac{2q}{q-1}}\left(-\frac{\log\delta}{2\log 2}+1\right)}\right)^{\frac{n}{n-1}\frac{q-1}{q}}\\
&\leq c|B_{2R}|\left(\dfrac{1}{\nu^{\frac{2q}{q-1}}|\log\delta|}\right)^{\frac{n}{n-1}\frac{q-1}{q}}.
\end{split}
\end{align}

Combining \eqref{eq:ke7}, \eqref{eq:ke9}, \eqref{eq:ke10} and \eqref{eq:ke11}, we conclude with
\begin{align*}
|B_{2R}\cap\{u\leq 2\delta t\}|\leq\dfrac{c}{\nu^{\max\left\{\frac{2q}{q-1},\frac{2n}{n-1}\right\}}}\left(\delta^{\frac{p-1}{2}}+\frac{1}{|\log\delta|^{\frac{n}{n-1}\frac{q-1}{q}}}\right)|B_{2R}|,
\end{align*}
which completes the proof.
\end{proof}

Using Lemma \ref{lem:key}, we now show the expansion of positivity.

\begin{lemma}\label{lem:grow}
Let $u \in \mathcal{A}(\Omega)\cap L^{p-1}_{sp}(\mathbb{R}^{n})$ be a minimizer of $\mathcal{E}$ which is nonnegative in a ball $B_{4R}\Subset \Omega$ with $R\leq 1$. When $sp \le n$, assume further that $u$ is bounded in $B_{4R}$. Suppose that
\begin{align*}
|B_{2R}\cap\{u\geq t\}|\geq\nu|B_{2R}|
\end{align*}
for some $\nu\in(0,1)$ and $t>0$. Then there exists $\delta=\delta(\data(B_{4R}),\nu)\in(0,\frac{1}{8}]$ such that if
\begin{align}\label{eq:as1}
\tail(u_-;4R)\leq g_{B_{4R}}(\delta t),
\end{align}
then we have $u\geq \delta t$ in $B_R$.
\end{lemma}

\begin{proof}
If $a^{-}_{B_{4R}}=0$, then \eqref{eq:as1} follows from \cite[Lemma 6.5]{C1}. Hence we only consider the case $a^{-}_{B_{4R}}>0$, in which $u\in W^{1,q}(B_{2R})$. Choose numbers $h,k$ such that $\delta t\leq h<k\leq 2\delta t$ and radii $\rho,r$ such that $2R\leq\rho<r\leq 4R$. 
Define
\begin{align*}
A^-(h,\rho)=B_{\rho}\cap\{u\leq h\}.
\end{align*}
Note that we can always choose
\begin{equation}\label{eq:exponent}
\kappa \coloneqq \frac{p^{*}_{s}}{p} < \frac{q^{*}}{q}.
\end{equation}
Indeed, if $q < n$, then we have $\kappa = n/(n-sp) < n/(n-q) = q^{*}/q$ from \eqref{eq:spq}. If $q \ge n$, then we can choose the number $q^{*}$ large enough to satisfy \eqref{eq:exponent}.

Recalling that $u\in W^{1,q}(B_{2R})$, we now apply Sobolev's embedding theorem and \eqref{eq:fsp} to have
\begin{equation*}
\begin{aligned}
\lefteqn{ \left[\left(\frac{k-h}{\rho^{s}}\right)^{p}+a^{-}_{B_{4R}}\left(\frac{k-h}{\rho}\right)^{q}\right]\left(\frac{|A^{-}(h,\rho)|}{|B_{\rho}|}\right)^{\frac{1}{\kappa}} }\\
& \le \left(\mean{B_{\rho}}\left[\left(\frac{(u-k)_{-}}{\rho^{s}}\right)^{p}+a^{-}_{B_{4R}}\left(\frac{(u-k)_{-}}{\rho}\right)^{q}\right]^{\kappa}\,dx\right)^{\frac{1}{\kappa}} \\
& \le c\mean{B_{\rho}}\int_{B_{\rho}}\frac{|(u(x)-k)_{-}-(u(y)-k)_{-}|^{p}}{|x-y|^{n+sp}}\,dxdy + c\mean{B_{\rho}}a^-_{B_{4R}}|D(u-k)_{-}|^{q}\,dx \\
& \quad + c\mean{B_{\rho}}\left(\frac{(u-k)_{-}}{\rho^{s}}\right)^{p} + a^{-}_{B_{4R}}\left(\frac{(u-k)_{-}}{\rho}\right)^{q}\,dx \\
& \le c\mean{B_{\rho}}\int_{B_{\rho}}\frac{|(u(x)-k)_{-}-(u(y)-k)_{-}|^{p}}{|x-y|^{n+sp}}\,dxdy + c\mean{B_{\rho}}a^-_{B_{4R}}|D(u-k)_{-}|^{q}\,dx \\
& \quad + c\left[\left(\frac{\delta t}{\rho^{s}}\right)^{p}+a^{-}_{B_{4R}}\left(\frac{\delta t}{\rho}\right)^{q}\right]\frac{|A^{-}(k,r)|}{|B_{r}|},
\end{aligned}
\end{equation*}
where we have also used the fact that $k\leq 2\delta t$.
Applying Lemma~\ref{lem:DG} to the right-hand side of the above display, and then recalling  \eqref{eq:Gg} and the fact that $\rho\in[2R,4R]$, we find
\begin{align}\label{eq:g4}
\begin{split}
\lefteqn{ G_{B_{4R}}(k-h) 
\left(\dfrac{|A^-(h,\rho)|}{|B_{\rho}|}\right)^{\frac{1}{\kappa}} }\\
&\leq \dfrac{c}{|B_r|}\left(\frac{r}{r-\rho}\right)^{n+q}\left[\dfrac{\|(u-k)_-\|^p_{L^p(B_r)}}{r^{sp}}+a^{-}_{B_{4R}}\dfrac{\|(u-k)_-\|^q_{L^q(B_r)}}{r^q}\right]\\
&\quad +\frac{c}{|B_{r}|}\left(\frac{r}{r-\rho}\right)^{n+q}\|(u-k)_-\|_{L^1(B_r)}\cdot \tail((u-k)_-;r)\\
&\quad+c G_{B_{4R}}(\delta t)\dfrac{|A^-(k,r)|}{|B_r|}.
\end{split}
\end{align}
For the right-hand side of the above inequality, note that
\begin{align}\label{eq:g4.1}
\|(u-k)_-\|^{\theta}_{L^{\theta}(B_r)}=\int_{A^-(k,r)}(k-u(x))^{\theta}\,dx\leq |A^-(k,r)|k^{\theta} 
\end{align}
for any $\theta\geq 1$. Also, since $r\in[2R,4R]$ and $k\in[\delta t,2\delta t]$, we discover 
\begin{align}\label{eq:g4.2}
\begin{split}
\tail((u-k)_-;r)&=\int_{\ern\setminus B_{r}}\dfrac{(k-u(x))^{p-1}_+}{|x|^{n+sp}}\,dx\\
&\leq c\left[k^{p-1}\int_{\ern\setminus B_{\rho}}\dfrac{dx}{|x|^{n+sp}}+\int_{\ern\setminus B_{4R}}\dfrac{u_-(x)^{p-1}}{|x|^{n+sp}}\,dx\right]\\
&=c\left[\rho^{-sp}k^{p-1}+\tail(u_-;4R)\right]\\
\overset{\eqref{eq:as1}}&{\leq} cg_{B_{4R}}(\delta t). 
\end{split}
\end{align}
Connecting \eqref{eq:g4.1} and \eqref{eq:g4.2} to \eqref{eq:g4}, we have
\begin{align*}
G_{B_{4R}}(k-h)\left(\dfrac{|A^-(h,\rho)|}{|B_{\rho}|}\right)^{\frac{1}{\kappa}}
\leq c\left(\frac{r}{r-\rho}\right)^{n+q}G_{B_{4R}}(k)\dfrac{|A^-(k,r)|}{|B_r|}
\end{align*}
and so 
\begin{align}\label{eq:g6}
\dfrac{|A^-(h,\rho)|}{|B_{\rho}|}
\leq c\left(\frac{r}{r-\rho}\right)^{(n+q)\kappa}\left(\frac{G_{B_{4R}}(k)}{G_{B_{4R}}(k-h)}\right)^{\kappa}\left(\dfrac{|A^-(k,r)|}{|B_{r}|}\right)^{\kappa}.
\end{align}

For $i\in\mathbb{N}\cup\{0\}$, define 
\[ r_i=(1+2^{-i})R, \quad  k_i=(1+2^{-i})\delta t \quad \text{and} \quad \phi_i\coloneqq\frac{|A^-(k_i,r_i)|}{|B_{r_i}|}. \] 
Accordingly, we apply \eqref{eq:g6} with the choices $h=k_i$, $k=k_{i-1}$, $\rho=r_i$ and $r=r_{i-1}$. Then, since $k_{i-1}-k_{i}=2^{-i}\delta t$ and $\frac{r_{i-1}}{r_{i-1}-r_{i}}\leq 2^{i}$, we arrive at
\begin{align*}
\phi_{i} \le c_{2}2^{i(n+2q)\kappa}\phi_{i-1}^{\kappa}
\end{align*}
with $c_{2}=c_{2}(\data(B_{4R}))$.

Choose $\delta=\delta(\data(B_{4R}),\nu)\in(0,\frac{1}{8}]$ such that
\begin{align*}
\tau \equiv \tau(\delta) \coloneqq \frac{c_1}{\nu^{\max\{\frac{2q}{q-1},\frac{2n}{n-1}\}}}\left(\delta^{\frac{p-1}{2}}+\frac{1}{|\log\delta|^{\frac{n}{n-1}\frac{q-1}{q}}}\right)\leq c_2^{-\frac{1}{1-\kappa}}2^{-(n+2q)\frac{\kappa}{(1-\kappa)^2}},
\end{align*} 
where $c_1$ is the constant determined in Lemma~\ref{lem:key}. Then we apply Lemma~\ref{lem:key} in order to have
\begin{align*}
\phi_0=\frac{|A^-(2\delta t,2R)|}{|B_{2R}|}
\leq c_{2}^{-1/(\kappa-1)}2^{-(n+2q)\kappa/(\kappa-1)^2}.
\end{align*}
Therefore, Lemma~\ref{lem:it} implies $\lim_{i\rightarrow\infty}\phi_i=0$, and we conclude that $u\geq \delta t$ in $B_{R}$.
\end{proof}

\section{Proof of Theorems~\ref{thm:hol} and \ref{thm:har}}\label{sec:5}
In this section, we prove Theorems~\ref{thm:hol} and \ref{thm:har}. As in the previous section, we assume that $K_{sp}$ is symmetric and satisfy \eqref{eq:ker}, $F$ satisfies \eqref{eq:F}, and $a(\cdot)$ satisfies \eqref{eq:a} and \eqref{eq:hol.a}. Also, we assume that $s,p,q$ and $\alpha$ satisfy \eqref{eq:spq} and \eqref{eq:qspa2}. 

\subsection{Proof of Theorem~\ref{thm:hol}} 
By translation, without loss of generality we assume $x_0=0$. Let $\delta\in(0,1/8]$ be the constant defined in Lemma \ref{lem:grow}. Choose $\gamma = \gamma(\data(B_{4R})) \in (0,1)$ such that 
\begin{align}\label{eq:c1}
0<\gamma\leq\min\left\{\frac{s}{2},\log_4\left(\dfrac{2}{2-\delta}\right)\right\}
\end{align}
and
\begin{align}\label{eq:c2}
\int^{\infty}_{4}\dfrac{(\rho^{\gamma}-1)^{p-1}}{\rho^{1+sp}}\,d\rho\leq\dfrac{s\delta^{p-1}}{8^{p+1}n|B_1|}.
\end{align}
Observe that the left-hand side of \eqref{eq:c2} is an increasing function of $\gamma$. Thus, if \eqref{eq:c2} holds for $\gamma$ determined in \eqref{eq:c1}, then \eqref{eq:c2} holds for any $\beta\leq \gamma$ as well. We choose the number
\begin{align}\label{eq:c3}
j_0 \coloneqq \left\lceil \dfrac{2}{sp}\log_4\left(\dfrac{2^{p}\left(1+n|B_1|/(sp)\right)}{\delta^{p-1}}\right) \right\rceil,
\end{align}
where $\lceil t \rceil$ denotes the least integer greater than or equal to $t$.

We will show that there exist a non-decreasing sequence $\{m_i\}$ and a non-increasing sequence $\{M_i\}$ such that for any $i\in\mathbb{N}\cup\{0\}$,
\begin{align}\label{eq:c4}
m_i\leq u\leq M_i\,\,\text{ in }\,\, B_{4^{1-i}R}\quad\text{and}\quad M_i-m_i=4^{-\gamma i}L,
\end{align}
where $L$ is defined as
\begin{align}\label{eq:c5}
L \coloneqq 2^{1+sj_{0}}\|u\|_{L^{\infty}(B_{4R})}+[(4R)^{sp}\text{Tail}(u;4R)]^{\frac{1}{p-1}}.
\end{align}
We use strong induction on $i$. Let $m_i \coloneqq -4^{-\gamma i}L/2$ and $M_i \coloneqq 4^{-\gamma i}L/2$ with $i=0,\dots,j_0$. Then from \eqref{eq:c1} and \eqref{eq:c5} we notice that \eqref{eq:c4} holds for $i=0,\dots,j_0$. Indeed, we have
\begin{align*}
m_i=-4^{-\gamma i}L/2\leq -4^{-si/2}2^{sj_0}\|u\|_{L^{\infty}(B_{4R})}=-2^{-si}2^{sj_0}\|u\|_{L^{\infty}(B_{4R})}\leq -\|u\|_{L^{\infty}(B_{4R})}\leq u(x)
\end{align*}
for a.e. $x \in B_{4^{1-i}R}$; a similar argument also shows that $u\leq M_i$ a.e. in $B_{4^{1-i}R}$. Now, we choose an integer $j\geq j_0$ and assume the sequences $\{m_i\}$ and $\{M_i\}$ are constructed for $i \in \{1,\ldots,j\}$. Then we are going to prove \eqref{eq:c4} for $i=j+1$, by constructing $m_{j+1}$ and $M_{j+1}$ properly. 

We start by observing that either
\begin{equation}\label{eq:case1}
\left|B_{4^{1-j}R/2}\cap\left\{u\geq m_{j}+\frac{M_{j}-m_{j}}{2}\right\}\right|\geq\frac{1}{2}\left|B_{4^{1-j}R/2}\right|
\end{equation}
or
\begin{equation}\label{eq:case2}
\left|B_{4^{1-j}R/2}\cap\left\{u\geq m_{j}+\frac{M_{j}-m_{j}}{2}\right\}\right|<\frac{1}{2}\left|B_{4^{1-j}R/2}\right|.
\end{equation}
We set 
\[ w \coloneqq 
\begin{cases}
u-m_{j} & \text{if } \eqref{eq:case1}\,\,\text{holds}, \\ 
M_{j}-u & \text{if } \eqref{eq:case2}\,\,\text{holds}
\end{cases} 
\]
and
\[ t = \frac{M_{j}-m_{j}}{2}. \]
In any case, $w$ satisfies \eqref{eq:D} and $w\geq 0$ in $B_{4^{1-j}R}$. Moreover, it holds that
\begin{align}\label{eq:c16}
\left|B_{4^{1-j}R/2}\cap\left\{w\geq t\right\}\right|\geq\frac{1}{2}\left|B_{4^{1-j}R/2}\right|.
\end{align}

We first consider the case \eqref{eq:case1}. 
Fix any $x\in B_{4R}\setminus B_{4^{1-j}R}$, and let $l\in\{0,\dots,j-1\}$ be the unique integer such that $x\in B_{4^{1-l}R}\setminus B_{4^{-l}R}$. Using \eqref{eq:c4} and the monotonicity of $\{m_i\}$, we have
\begin{align*}
w(x) = u - m_j  \ge m_l-M_l +2t = -2t\left(4^{(j-l)\gamma}-1\right) \ge -2t\left[\left(\frac{4^{j}|x|}{R}\right)^{\gamma}-1\right]
\end{align*}
for a.e. $x\in B_{4R}\setminus B_{4^{1-j}R}$.
Meanwhile, from \eqref{eq:c5} we immediately have
\begin{align*}
w(x) \ge -|u(x)| -\frac{L}{2} \quad\text{for a.e. } x \in \ern\setminus B_{4R}.
\end{align*}
Using the above two inequalities and change of variables, we estimate
\begin{equation*}
\begin{aligned}
\lefteqn{ \tail (w_-;4^{1-j}R) }\\
&\leq (2t)^{p-1}\int_{B_{4R}\setminus B_{4^{1-j}R}}\dfrac{[(4^{j}|x|/R)^{\gamma}-1]^{p-1}}{|x|^{n+sp}}\,dx + 2^{p-1}\int_{\ern\setminus B_{4R}}\dfrac{|u(x)|^{p-1}+L^{p-1}}{|x|^{n+sp}}\,dx \\
&\leq (2t)^{p-1}\frac{4^{sp}n|B_{1}|}{(4^{1-j}R)^{sp}}\int_{4}^{\infty}\frac{(\rho^{\gamma}-1)^{p-1}}{\rho^{1+sp}}\,d\rho + 2^{p-1}\left[ \tail(u;4R) + \frac{n|B_{1}|}{sp}\frac{L^{p-1}}{(4R)^{sp}}\right].
\end{aligned}
\end{equation*}
The first term in the right-hand side is further estimated by using \eqref{eq:c2}:
\begin{align*}
(2t)^{p-1}\frac{4^{sp}n|B_{1}|}{(4^{1-j}R)^{sp}}\int^{\infty}_{4}\dfrac{(\rho^{\gamma}-1)^{p-1}}{\rho^{1+sp}}\,d\rho \leq \frac{1}{2}\frac{(\delta t)^{p-1}}{(4^{1-j}R)^{sp}} \le \frac{1}{2}g_{B_{4^{1-j}R}}(\delta t).
\end{align*}
For the second term, note from \eqref{eq:c1} that $\gamma\leq\frac{s}{2}<\frac{sp}{2(p-1)}$. Then in light of \eqref{eq:D}, \eqref{eq:c1}, \eqref{eq:c3} and \eqref{eq:c5}, we find
\begin{equation*}
\begin{aligned}
2^{p-1}\left[ \tail(u;4R) + \frac{n|B_{1}|}{sp}\dfrac{L^{p-1}}{(4R)^{sp}}\right] & \le 2^{p-1}\left(1+\frac{n|B_{1}|}{sp}\right)\frac{L^{p-1}}{(4R)^{sp}} \\
& = 2^{p-1}4^{\gamma(p-1)j-spj}\left(1+\frac{n|B_{1}|}{sp}\right)\frac{t^{p-1}}{(4^{1-j}R)^{sp}} \\
& \le \frac{1}{2}\frac{(\delta t)^{p-1}}{(4^{1-j}R)^{sp}} \le \frac{1}{2}g_{B_{4^{1-j}R}}(\delta t).
\end{aligned}
\end{equation*}
Combining the last three displays gives
\[ \tail(w_{-};4^{1-j}R) \le g_{B_{4^{1-j}R}}(\delta t). \]
With this and \eqref{eq:c16} at hand, we can apply Lemma \ref{lem:grow} to $w$, which gives
\begin{align*}
w\geq \delta t \quad\text{in }B_{4^{-j}R}.
\end{align*}
Summarizing, in the case \eqref{eq:case1}, we use  \eqref{eq:c1} and \eqref{eq:c4} together with the above estimates to obtain
\begin{align*}
u \ge m_{j} + \delta t  = m_{j} + \frac{\delta 4^{-\gamma j}L}{2} \geq m_{j} + 4^{-\gamma j}(1-4^{-\gamma})L \quad\text{in}\,\,B_{4^{-j}R}.
\end{align*}
By setting $M_{j+1} \coloneqq M_j$ and $m_{j+1} \coloneqq m_{j} + 4^{-\gamma j}(1-4^{-\gamma})L$, \eqref{eq:c4} is true for $i=j+1$. 

In the remaining case \eqref{eq:case2}, we can argue in a completely similar way, this time with $m_{j+1} \coloneqq m_j$ and $M_{j+1} \coloneqq M_{j} - 4^{-\gamma j}(1-4^{-\gamma})L$, to obtain \eqref{eq:c4} for $i=j+1$. By strong induction on $i$, we conclude that \eqref{eq:c4} holds for any $i\in\mathbb{N}\cup\{0\}$, and H\"{o}lder continuity of $u$ follows in a standard way.
\, \hfill \qed

\subsection{Harnack's inequality}
We next prove Harnack's inequality in Theorem~\ref{thm:har}. 
The following lemma can be proved in a very similar way as in Lemma~\ref{lem:grow}.

\begin{lemma}\label{lem:grow*}
Let $u \in \mathcal{A}(\Omega) \cap L^{p-1}_{sp}(\mathbb{R}^{n})$ be a minimizer of $\mathcal{E}$ which is nonnegative in a ball $B_{16R}\Subset \Omega$ with $R\leq 1$. When $sp \le n$, assume further that $u$ is bounded in $B_{16R}$.
Suppose that
\begin{align*}
|B_{R}\cap\{u\geq t\}|\geq\nu^k|B_{R}|
\end{align*}
for some $t>0$ and $\nu\in(0,1)$. Then there exists $\delta=\delta(\data(B_{16R}),\nu)\in(0,\frac{1}{8}]$ such that if
\begin{align*}
\tail(u_-;16R)\leq g_{B_{16R}}(\delta^k t),
\end{align*}
then we have $u\geq \delta^k t$ in $B_R$.
\end{lemma}

Using Lemma \ref{lem:grow*}, we have the following:

\begin{lemma}\label{lem:wh}
Let $u \in \mathcal{A}(\Omega) \cap L^{p-1}_{sp}(\mathbb{R}^{n})$ be a minimzer of $\mathcal{E}$ which is nonnegative in a ball $B_{16R}\Subset \Omega$ with $R\leq 1$. When $sp \le n$, assume further that $u$ is bounded in $B_{16R}$. Then there exist constants $\varepsilon_{0}\in(0,1)$ and $c\geq 1$, both depending on $\data(B_{16R})$, such that
\begin{align}\label{eq:wh}
\left(\mean{B_{R}}u^{\varepsilon_{0}}\,dx\right)^{\frac{1}{\varepsilon_{0}}}\leq c\inf_{B_R}u + cg^{-1}_{B_{16R}}(\tail(u_-;16R)).
\end{align}
\end{lemma}
\begin{proof}
We assume that $u$ does not vanish on $B_{R}$, otherwise there is nothing to prove.
Let $\delta\in(0,\frac{1}{8}]$ be the constant determined in Lemma \ref{lem:grow*} with the choice $\nu=\frac{1}{2}$. We accordingly set
\begin{align}\label{eq:wh1}
\varepsilon_{0} \coloneqq \frac{\log\nu}{2\log\delta} = \frac{1}{2\log_{\frac{1}{2}}\delta} \in(0,1).
\end{align}	
We claim that for any $t\geq 0$,
\begin{align}\label{eq:wh2}
\inf_{B_R}u+g^{-1}_{B_{16R}}(\tail(u_-;16R))\geq\delta\left(\dfrac{|A^+(t,R)|}{|B_R|}\right)^{\frac{1}{2\varepsilon_{0}}}t.
\end{align}
We only consider the case $t\in[0,\sup_{B_R}u)$, otherwise \eqref{eq:wh2} is trivial.
	
For each $t\in[0,\sup_{B_R}u)$, let $k=k(t)$ be the unique integer satisfying
\begin{align}\label{eq:wh4}
\log_{\frac{1}{2}}\dfrac{|A^+(t,R)|}{|B_R|}\leq k<1+\log_{\frac{1}{2}}\dfrac{|A^+(t,R)|}{|B_R|}.
\end{align}	
Then \eqref{eq:wh1} and \eqref{eq:wh4} imply
\begin{align}\label{eq:wh5}
\delta^k\geq\delta\left(\dfrac{|A^+(t,R)|}{|B_R|}\right)^{\frac{1}{2\varepsilon_{0}}}.
\end{align}
We assume that 
\begin{align*}
\quad\tail(u_-;16R)<g_{B_{16R}}(\delta^kt),
\end{align*}
otherwise \eqref{eq:wh2} again follows directly. Now, observe that \eqref{eq:wh4} implies
\begin{align*}
|A^+(t,R)|\geq 2^{-k}|B_R|.
\end{align*}	
Then we are in a position to apply Lemma~\ref{lem:grow*}, which gives
\begin{align*}
u\geq\delta^k t\quad\text{in }B_{R},
\end{align*}
and so
\begin{align}\label{eq:wh6}
\inf_{B_R}u+g^{-1}_{B_{16R}}(\tail(u_-;16R))\geq\delta^{k}t.
\end{align}
Combining \eqref{eq:wh6} and \eqref{eq:wh5}, we have \eqref{eq:wh2}. 
At this moment, a similar argument as in the proof of \cite[Proposition~6.8]{C1} yields \eqref{eq:wh}.
\end{proof}
Now we have the following local sup-estimate.
\begin{lemma}\label{lem:bdd}
Let $u \in \mathcal{A}(\Omega) \cap L^{p-1}_{sp}(\mathbb{R}^{n})$ be a minimizer of $\mathcal{E}$ and $B_{2r} \equiv B_{2r}(z) \Subset \Omega$ a ball. When $sp \le n$, assume further that $u$ is bounded in $B_{2r}$. Then for any $\delta\in(0,1)$, we have
\begin{align}\label{eq:bdd}
\sup_{B_{r}}u_+\leq c_{\delta}G^{-1}_{B_{2r}}\left(\mean{B_{2r}}G_{B_{2r}}(u_+)\,dx\right)+\delta\,g^{-1}_{B_{r}}\left(\tail(u_+;r)\right)
\end{align}
with $c=c(\data(B_{2r}))$ and $c_{\delta}=c_{\delta}(\data(B_{2r}),\delta)$.
\end{lemma}

\begin{proof}
For any $j\in\mathbb{N}\cup\{0\}$, we write
\begin{align*}
r_j= (1+2^{-j})r, \quad B_j=B_{r_j}, \quad k_j=(1-2^{-j-1})2k_{0}, \quad w_j=(u-k_j)_+.
\end{align*}
Observe that
\[ r < r_{j+1} < r_{j} < 2r, \quad k_j < k_{j+1}, \quad w_{j+1}\leq w_j. \]

By using Lemma \ref{lem:DG}, we have
\begin{align}\label{eq:bdd2.2}
\begin{split}
\lefteqn{ [w_{j+1}]^p_{W^{s,p}(B_{j+1})}+a^{-}_{B_{2r}}[w_{j+1}]^q_{W^{1,q}(B_{j+1})} }\\
&\leq c\left(\frac{r_{j}}{r_{j}-r_{j+1}}\right)^{n+q}\left[\int_{B_{j}}\left(\frac{w_j}{r_j^s}\right)^p\,dx+a^-_{B_{2r}}\int_{B_{j}}\left(\frac{w_j}{r_j}\right)^q\,dx+\|w_j\|_{L^1(B_{j})}(\tail(w_j;r_j))\right]\\
&\leq c\left(\frac{r_{j}}{r_{j}-r_{j+1}}\right)^{n+q}\left[\int_{B_{j}}G_{B_{2r}}(w_j)\,dx+\left(\int_{B_{j}}w_j\,dx\right)(\tail(w_j;r_j))\right]
\end{split}
\end{align}
for a constant $c=c(\data(B_{2r}))$, where we have also used the relation that 
\begin{align*}
r_{j}\eqsim r\quad\Rightarrow\quad G_{B_{2r}}(t)\eqsim \left(\frac{t}{r^s_{j}}\right)^p+a^{-}_{B_{2r}}\left(\frac{t}{r_{j}}\right)^q\quad \forall\; t\geq 0.
\end{align*}
Now, with $\kappa$ defined in \eqref{eq:exponent}, we use \eqref{eq:fsp} and Sobolev's embedding theorem to find
\begin{align*}
\begin{split}
\lefteqn{ \mean{B_{j+1}}G_{B_{2r}}(w_{j+1})\,dx \le \left(\frac{|A^{+}(k_{j+1},r_{j+1})|}{|B_{j+1}|}\right)^{\frac{1}{\kappa'}}\left(\mean{B_{j+1}}[G_{B_{2r}}(w_{j+1})]^{\kappa}\,dx\right)^{\frac{1}{\kappa}} }\\
& \le c\left(\frac{|A^{+}(k_{j+1},r_{j+1})|}{|B_{j+1}|}\right)^{\frac{1}{\kappa'}}\left(\mean{B_{j+1}}\int_{B_{j+1}}\frac{|w_{j+1}(x)-w_{j+1}(y)|^{p}}{|x-y|^{n+sp}}\,dxdy + a^{-}_{B_{2r}}\mean{B_{j+1}}|Dw_{j+1}|^{q}\,dx\right) \\
& \quad + c\left(\frac{|A^{+}(k_{j+1},r_{j+1})|}{|B_{j+1}|}\right)^{\frac{1}{\kappa'}}\mean{B_{j+1}}G_{B_{2r}}(w_{j+1})\,dx.
\end{split}
\end{align*}
We also observe that 
\begin{equation}\label{eq:bdd2.6}
\begin{aligned}
|A^{+}(k_{j+1},r_{j+1})| & \le \frac{1}{G_{B_{2r}}(k_{j+1}-k_{j})}\int_{A^{+}(k_{j},r_{j})}G_{B_{2r}}(w_{j})\,dx, \\
\mean{B_{j+1}}w_{j+1}\,dx & \le \frac{1}{g_{B_{2r}}(k_{j+1}-k_{j})}\mean{B_{j}}G_{B_{2r}}(w_{j})\,dx.
\end{aligned}
\end{equation}
Combining \eqref{eq:bdd2.2}-\eqref{eq:bdd2.6}, we find 
\begin{align}\label{eq:bdd2.6.1}
\begin{split}
\mean{B_{j+1}}G_{B_{2r}}(w_{j+1})\,dx & \le \frac{c}{[G_{B_{2r}}(k_{j+1}-k_{j})]^{1/\kappa'}}\left(\frac{r_{j}}{r_{j}-r_{j+1}}\right)^{n+q}\left(1+\frac{\tail(w_{j};r_{j})}{g_{B_{2r}}(k_{j+1}-k_{j})}\right)\\
& \quad \cdot\left(\mean{B_{j}}G_{B_{2r}}(w_{j})\,dx\right)^{1+\frac{1}{\kappa'}}.
\end{split}
\end{align}

Denoting
\begin{align*}
a_j \coloneqq \dfrac{1}{|B_{r}|}\int_{A^{+}(k_{j},r_{j})}G_{B_{2r}}(w_j)\,dx,
\end{align*}
and recalling the definitions of $k_{j}$ and $r_{j}$, we see that \eqref{eq:bdd2.6.1} becomes
\begin{align*}
a_{j+1}\leq \frac{c2^{(n+q)j}}{[G_{B_{2r}}(2^{-j}k_{0})]^{1/\kappa'}}\left(1+\frac{\tail(u_{+};r)}{g_{B_{2r}}(2^{-j}k_{0})}\right)a_{j}^{1+\frac{1}{\kappa'}}.
\end{align*}
Here, if $k_{0}$ is so large that
\begin{equation}\label{eq:k0.1}
\frac{\tail(u_{+};r)}{g_{B_{2r}}(k_{0}/\delta)} \le \frac{\delta^{p}[\tail(u_{+};r)]}{g_{B_{2r}}(k_{0})} \le 1,
\end{equation}
then
\begin{align*}
a_{j+1}\leq\frac{c_{2}2^{\theta j}}{\delta^{p}[G_{B_{2r}}(k_{0})]^{1/\kappa}}a_{j}^{1+\frac{1}{\kappa'}}
\end{align*}
holds for a constant $c_{2} = c_{2}(\data(B_{2r}))$, where
\[ \theta \coloneqq \frac{p}{\kappa'} + n+p+q-1. \]
We now fix
\[ k_{0} = G_{B_{2r}}^{-1}\left[\left(\frac{c_{2}}{\delta^{p}}\right)^{\kappa'}2^{\theta(\kappa')^{2}}\mean{B_{2r}}G_{B_{2r}}(u_{+})\,dx \right] + \delta g_{B_{2r}}^{-1}\left( \tail(u_{+};r)\right). \]
Then \eqref{eq:k0.1} holds, and moreover we can apply Lemma~\ref{lem:it} to conclude that $a_j\rightarrow 0$ as $j\rightarrow\infty$. In turn, an elementary manipulation gives the desired estimate \eqref{eq:bdd}.
\end{proof}

We are now ready to prove Theorem~\ref{thm:har}.
\begin{proof}[Proof of Theorem~\ref{thm:har}]
By translation, we assume that $x_0$ is the origin. 

\textit{Step 1: Tail estimates.} First, we claim that for any $z\in B_{R}$ and $0<r\leq 2R$,
\begin{align}\label{eq:h3}
\tail(u_+;z,r)&\leq cg_{B_r(z)}\left(\sup_{B_r(z)}u\right)+c\tail(u_-;z,r)
\end{align}
holds for a constant $c=c(\data(B_{2R}))$. Indeed, denoting $M \coloneqq \sup_{B_r(z)}u$, we apply \eqref{eq:D} with $k \equiv 2M$ to have
\begin{align}\label{eq:h4}
\begin{split}
I_1 & \coloneqq\int_{B_{r/2}(z)}(u(x)-2M)_-\left[\int_{\ern}\dfrac{(u(y)-2M)_+^{p-1}}{|x-y|^{n+sp}}\,dy\right]\,dx\\
&\leq c\left[ \dfrac{\|(u-2M)_-\|^p_{L^p(B_r(z))}}{r^{sp}}+ a^{-}_{B_r(z)}\dfrac{\|(u-2M)_-\|^q_{L^q(B_r(z))}}{r^q} \right]\\
&\quad+ c\|(u-2M)_-\|_{L^1(B_r(z))}\tail((u-2M)_-;z,r/2) \\
& \eqqcolon I_{2}.
\end{split}
\end{align}
For $I_1$, we first notice that 
\begin{align*}
|x-y|\leq 2|y-z|\quad\text{for any}\quad x\in B_{r}(z)\quad\text{and}\quad y\in\ern\setminus B_{r}(z).
\end{align*}
Also, from \cite[Lemma~4.4]{C1} we obtain
\begin{align*}
(u(y)-2M)^{p-1}_+\geq\min\{1,2^{2-p}\}u_+(y)^{p-1}-2^{p-1}M^{p-1}.
\end{align*}
From the above two observations and the fact that $u\leq M$ on $B_r(z)$, it follows that
\begin{align*}
\begin{split}
\lefteqn{ \int_{B_{\frac{r}{2}}(z)}(u(x)-2M)_-\left[\int_{\ern}\dfrac{(u(x)-2M)^{p-1}_+}{|x-y|^{n+sp}}\,dy\right]\,dx }\\
&\geq 2^{-n-sp}M\int_{B_{\frac{r}{2}}(z)}\left[\int_{\ern\setminus B_r(z)}\dfrac{\min\{1,2^{2-p}\}u_+(y)^{p-1}-2^{p-1}M^{p-1}}{|y-z|^{n+sp}}\,dy\right]\,dx\\
&\geq\dfrac{Mr^{n-sp}}{c}\tail(u_+;z,r)-cr^{n-sp}M^p.
\end{split}
\end{align*}
On the other hand, since $u\geq0$ on $B_r(z)$, we have
\begin{align*}
I_2\leq cr^{n-sp}\left(M^p+a^-_{B_r(z)}r^{sp-q}M^q+M\tail(u_-;z,r)\right).
\end{align*}
Merging the above two estimates together with \eqref{eq:h4} directly gives \eqref{eq:h3} as follows:
\begin{align*}
\tail(u_+;z,r)&\leq cr^{sp}\left(\frac{M^{p-1}}{r^{sp}}+a^-_{B_r(z)}\frac{M^{q-1}}{r^{q}}+\dfrac{1}{r^{sp}}\tail(u_-;z,r)\right)\\
&\leq c\left(g_{B_r(z)}(M)+\tail(u_-;z,r)\right).
\end{align*}

\textit{Step 2: Proof of \eqref{eq:har}.}
With $\delta_1\in(0,1]$ being any number, we use Lemma~\ref{lem:bdd} to have
\begin{align*}
\sup_{B_{r}(z)}u\leq c_{\delta_1}G^{-1}_{B_r(z)}\left(\mean{B_{2r}(z)}G_{B_r(z)}(u)\,dx\right)+\delta_1g^{-1}_{B_r(z)}(\tail(u_+;z,r)),
\end{align*}
where $c=c(\data(B_{2R}))$. Combining this estimate with \eqref{eq:h3}, we find
\begin{align}\label{eq:h9}
\begin{split}
\sup_{B_{r}(z)}u&\leq c_{\delta_1}G^{-1}_{B_r(z)}\left(\mean{B_{2r}(z)}G_{B_r(z)}(u)\,dx\right)+\delta_1g^{-1}_{B_r(z)}\left(g_{B_r(z)}\left(\sup_{B_r(z)}u\right)+\tail(u_-;z,r)\right)\\
&\leq c_{\delta_1}G^{-1}_{B_r(z)}\left(\mean{B_{2r}(z)}G_{B_r(z)}(u)\,dx\right)+c\delta_1\left(\sup_{B_r(z)}u+g^{-1}_{B_r(z)}(\tail(u_-;z,r))\right).
\end{split}
\end{align}
We next recall the exponent $\varepsilon_{0} \in (0,1)$ determined in \eqref{eq:wh1}. Using Jensen's inequality with the convex function $t \mapsto [G_{B_{r}(z)}^{-1}(t)]^{q}$, and then Young inequality with conjugate exponents $q/(q-\varepsilon_{0})$ and $q/\varepsilon_{0}$, where $\ep_0$ is determined in Lemma \ref{lem:wh}, we obtain
\begin{align}\label{eq:h10}
\begin{split}
G^{-1}_{B_{r}(z)}\left(\mean{B_{2r}(z)}G_{B_{r}(z)}(u)\,dx\right) & \le \left(\mean{B_{2r}(z)}u^{q}\,dx\right)^{\frac{1}{q}} \\
& \le \left(\sup_{B_{2r}(z)}u\right)^{\frac{q-\varepsilon_{0}}{q}}\left(\mean{B_{2r}(z)}u^{\varepsilon_{0}}\,dx\right)^{\frac{1}{q}} \\
& \le \delta_{2}\sup_{B_{2r}(z)}u + c_{\delta_{2}}\left(\mean{B_{2r}(z)}u^{\varepsilon_{0}}\,dx\right)^{\frac{1}{\varepsilon_{0}}}
\end{split}
\end{align}
for any $\delta_2>0$. 
Combining \eqref{eq:h9} and \eqref{eq:h10} and taking $\delta_1,\delta_2$ sufficiently small, we obtain
\begin{align}\label{eq:h11}
\sup_{B_{r}(z)}u\leq\frac{1}{2}\sup_{B_{2r}(z)}u+c\left(\mean{B_{2r}(z)}u^{\ep_{0}}\,dx\right)^{\frac{1}{\ep_{0}}}+c[r^{sp}\tail(u_-;z,r)]^{\frac{1}{p-1}},
\end{align}
where we have also used the fact that $g_{B_{r}(z)}^{-1}(t) \le (r^{sp}t)^{\frac{1}{p-1}}$ for any $t \ge 0$.

Now, let $R\leq\rho<\tau\leq 2R$ be fixed. 
By employing \eqref{eq:h11} along with a suitable covering argument, we arrive at
\begin{align*}
\sup_{B_{\rho}}u 
&\leq \frac{1}{2}\sup_{B_{\tau}}u+\dfrac{c}{(\tau-\rho)^{n/q}}\|u\|_{L^{\ep_{0}}(B_{2R})}+c[R^{sp}\tail(u_-;R)]^{\frac{1}{p-1}}.
\end{align*}
Then an application of the technical lemma \cite[Lemma 4.11]{C1} gives
\begin{align*}
\sup_{B_{R}}u\leq c\left(\mean{B_{2R}}u^{\varepsilon_{0}}\,dx\right)^{\frac{1}{\varepsilon_{0}}}+c\left[R^{sp}\tail(u_-;R)\right]^{\frac{1}{p-1}},
\end{align*}
which with \eqref{eq:wh} yields the desired Harnack's inequality \eqref{eq:har}.
\end{proof}

\end{document}